\documentclass[reqno,twoside,11pt]{amsart}

\usepackage{amsmath, amsfonts, amssymb, esint, amsthm, nicefrac, bbm, dsfont, enumitem,}
\usepackage{epsfig, graphicx}
\usepackage{hyperref, cite}
\usepackage{xcolor} 

\def\sideremark#1{\ifvmode\leavevmode\fi\vadjust{\vbox to0pt{\vss
 \hbox to 0pt{\hskip\hsize\hskip1em
\vbox{\hsize2cm\small\raggedright\pretolerance10000
 \noindent #1\hfill}\hss}\vbox to8pt{\vfil}\vss}}}

\setlist[itemize]{leftmargin=*}
\DeclareMathAlphabet{\mathpzc}{OT1}{pzc}{m}{it}
\newcommand{\p}{\mathrm{\bf p}}
\newcommand{\R}{\mathbb{R}}
\newcommand{\PP}{\mathbb{P}}

\newcommand{\N}{\mathbb{N}}
\newcommand{\A}{\mathcal{A}}
\newcommand{\D}{\mathcal{D}}
\newcommand{\One}{\mathds{1}}
\newcommand{\Lp}{L_{u}}
\newcommand{\Lse}{\mathcal{L}_{s,\p}^\epsilon[u]}
\newcommand{\Ls}{\mathcal{L}_{s,\p}[u]}
\newcommand{\tLs}{\tilde{\mathcal{L}}_{s,\p}[u]}

\newcommand{\dmN}{\;\mathrm{d}\mu_s^N(z)}
\def\endproof{\hspace*{\fill}\mbox{\ \rule{.1in}{.1in}}\medskip }

\newcommand{\one}{\mathds{1}}
\renewcommand{\epsilon}{\varepsilon}

\numberwithin{equation}{section}

\theoremstyle{plain}
\newtheorem{theorem}{Theorem}[section]
\newtheorem{lemma}[theorem]{Lemma}
\newtheorem{corollary}[theorem]{Corollary}
\newtheorem{proposition}[theorem]{Proposition}

\theoremstyle{definition}
\newtheorem{remark}[theorem]{Remark}
\newtheorem{definition}[theorem]{Definition}

\begin{document}
\title[Tug-of-War for fractional $\p$-Laplacian]
{Non-local tug-of-war with noise for the geometric fractional $\p$-Laplacian} 
\author{Marta Lewicka}
\address{M. Lewicka: University of Pittsburgh, Department of
  Mathematics, 139 University Place, Pittsburgh, PA 15260} 
\email{lewicka@pitt.edu} 
	
\thanks{Marta Lewicka was supported by NSF grants DMS-1613153 and DMS-2006439.}

\begin{abstract}
This paper concerns the fractional $\p$-Laplace
operator $\Delta_\p^s$ in non-divergence form, which has been introduced in \cite{BCF2}. 
For any $\p\in [2,\infty)$ and $s\in (\frac{1}{2},1)$ we first define
two families of non-local, non-linear averaging operators,
parametrised by $\epsilon$ and defined for all
bounded, Borel functions $u:\R^N\to \R$. We prove that $\Delta_\p^s u(x)$
emerges as the $\epsilon^{2s}$-order coefficient in the expansion of the deviation
of each $\epsilon$-average from the value $u(x)$, in the limit of the domain of
averaging exhausting an appropriate cone in $\R^N$ at the rate $\epsilon\to 0$.

Second, we consider the $\epsilon$-dynamic
programming principles modeled on the first average,  
and show that their solutions converge uniformly as $\epsilon\to 0$, to viscosity
solutions of the homogeneous non-local Dirichlet problem for
$\Delta_\p^s$, when posed in a domain $\D$ that satisfies the external cone condition and
subject to bounded, uniformly continuous data on $\R^N\setminus \D$.

Finally, we interpret such $\epsilon$-approximating solutions as values to the 
non-local Tug-of-War game with noise. In this game, players choose
directions while the game position is updated randomly within the
infinite cone that aligns with the specified direction, whose
aperture angle depends on $\p$ and $N$, and whose $\epsilon$-tip has been removed.
\end{abstract}

\maketitle

\section{Introduction}

This paper concerns a version of the fractional $\p$-Laplace
operator, which has been introduced in \cite{BCF2}. More precisely, for
$\p\geq 2$, $s\in (\frac{1}{2},1)$, and for a given bounded function
$u:\R^N\to\R$ that is of regularity $C^{1,1}(x)$ with $\nabla u(x)\neq 0$, one defines:
\begin{equation}\label{dps}
\Delta_\p^s u(x) \doteq C_{N,\p,s} \int_{T_\p^{0,\infty}(\frac{\nabla u(x)}{|\nabla
  u(x)|})}\frac{u(x+z)+u(x-z) - 2u(x)}{|z|^{N+2s}}\;\mbox{d}z.
\end{equation}
Above, $C_{N,\p,s}$ is a specific constant depending on $N,\p,s$, whereas the integration occurs on the infinite cone
$T_\p^{0,\infty}(\frac{\nabla u(x)}{|\nabla u(x)|})\subset\R^N$ whose
centerline is aligned with the vector $\frac{\nabla u(x)}{|\nabla u(x)|}$ and
whose aperture angle $\alpha$ depends on $N,\p$. In particular, for $\p=2$
we have $\alpha=\frac{\pi}{2}$ so that the said cone becomes the
half-space and (\ref{dps}) is consistent with the familiar formula: $-(-\Delta)^s u(x) = C_{N,s}\int_{\R^N}
\frac{u(z)- u(x)}{|x-z|^{N+2s}}\;\mbox{d}z$.
On the other hand, when $\p\to\infty$  then $\alpha\to 0$ and the
cone reduces to a line, consistently with the parallel
definition for fractional infinity Laplacian $\Delta_\infty^su(x)$ in \cite{BCF}.

As pointed out in \cite{BCF2}, definition (\ref{dps}) arises naturally
when extending the game-theoretical interpretation to the non-local, non-divergence
version of the classical $\p$-Laplace operator $\Delta_\p$. The
interpretation for $\Delta_\p$ 
has been originally put forward in \cite{PS} and it is based on the Tug-of-War game with random
noise, which in its turn can be seen as the interpolation between the pure Tug-of-War developed for the
$\infty$-Laplacian $\Delta_\infty$ in \cite{PSSW}, and the 
random walk description of the linear harmonic operator $\Delta$, which is classical.
In order to emphasise the importance of the choice of the integration cone $T_\p^{0,\infty}$ and
to distinguish the formula (\ref{dps}) from the divergence form of the
fractional $\p$-Laplacian arising through the Euler-Lagrange equations
of an appropriate non-local energy \cite{CLM}, we call the operator
$\Delta_\p^s$ above the ``geometric'' $\p$-$s$-Laplacian.

\smallskip

The purpose of this paper is to rigorously define the non-local
version of the noisy Tug-of-War game and prove that its
values converge to viscosity solutions of the
Dirichlet problem for $\Delta_\p^s$, posed on a sufficiently regular domain $\D\subset\R^N$:
\begin{equation}\label{ndci}
\Delta_\p^su = 0 ~\mbox{ in }\; \D,\qquad u=F ~\mbox{ in }\; \R^N\setminus\D.
\end{equation}
We remark that condition $\p\geq 2$ which we assume throughout, can be relaxed to
cover the full range $\p\in (1,\infty)$, by replacing the cone $T_\p$
with the complement of its doubled version for $\p\in (1,2)$. This
construction has been proposed in \cite[Remark 4.5]{BCF2}.
We now describe our main results. The said game will be modeled on
the dynamic programming principle that involves an appropriate
average, in whose asymptotic expansion the operator $\Delta_\p^s$
arises as the highest order term, in the vanishing limit of the expansion
parameter $\epsilon$. Hence, our first set of results develops such asymptotic
expansions, reminiscent of the well known local and linear formula:
\begin{equation}\label{aed2}
\fint_{B_\epsilon(x)}u(y)\;\mbox{d}y = u(x) +
\frac{\epsilon^2}{2(N+2)}\Delta u(x) + o(\epsilon^2) \qquad \mbox{ as }\; \epsilon\to 0+. 
\end{equation}

\smallskip

\subsection{Asymptotic expansions.} 
More precisely, we define the following non-local and non-linear averaging operator:
$$\A_\epsilon u(x) \doteq \frac{1}{2} \Big(\sup_{|y|=1}\fint_{T_\p^{\epsilon, \infty}(y)}
\frac{u(x+z)}{|z|^{N+2s}}\;\mbox{d}z + \inf_{|y|=1}\fint_{T_\p^{\epsilon, \infty}(y)}
\frac{u(x+z)}{|z|^{N+2s}} \;\mbox{d}z \Big),$$
where the integration takes place on the truncated infinite cones
$T_\p^{\epsilon, \infty}(y)=T_\p^{0, \infty}(y)\setminus
B_\epsilon(0)$, each oriented along its indicated unit direction vector $y$ and
having the aperture angle $\alpha$ as in (\ref{dps}). The integral averages $\fint$ are taken with respect to the
singular measure $|z|^{-N-2s}\,\mbox{d}z$. 
Note that $\mathcal{A}_\epsilon u$ is well defined for any bounded, Borel
function $u$, and in particular it does not necessitate the existence or the knowledge
of $\nabla u(x)$, which was essential in (\ref{dps}). 
The form of $\mathcal{A}_\epsilon$ is justified by the following expansion,
which we prove to be valid for functions $u$ that are $C^2$ in the vicinity of
a given $x\in\R^N$ with $\nabla u(x)\neq 0$, and uniformly continuous away from $x$:
\begin{equation}\label{raz}
\begin{split}
\A_\epsilon u(x) = u(x) + \frac{s}{(2-2s)(N+\p-2)} \epsilon^{2s}\cdot \Delta_\p^s u (x)
+ o(\epsilon^{2s}) \qquad \mbox{ as }\; \epsilon\to 0+. 
\end{split}
\end{equation}

We also propose another nonlinear average of a combined local - non-local nature:
\begin{equation*}
\begin{split}
\bar\A_\epsilon u(x)\doteq &\; \frac{(1-s)(N+\p-2)}{N+\p-2+2s}\cdot
\A_\epsilon u(x) \\ & + \frac{s(N+2)}{N+\p-2+2s}\cdot \fint_{B_\epsilon (x)} u(y)\;\mathrm{d}y
+ \frac{s(\p-2)}{N+\p-2+2s}\cdot\frac{1}{2} \Big(\sup_{B_\epsilon (x)}u + \inf_{B_\epsilon (x)}u \Big).
\end{split}
\end{equation*}
Note that the three positive multiplication factors above add up to
$1$. We prove that the result as in (\ref{raz}) similarly holds for $\bar{\mathcal{A}}_\epsilon$:
\begin{equation}\label{dwa}
\bar\A_\epsilon u(x) = u(x) + \frac{s}{2(N+\p-2+2s)}\epsilon^{2s}\cdot
\Delta_\p^s u(x) + o(\epsilon^{2s})\qquad \mbox{ as }\; \epsilon\to 0+.
\end{equation}
Expansion (\ref{dwa})  is superior to (\ref{raz}), because
the error quantity $o(\epsilon^{2s})$ in (\ref{raz}), which we make precise in the paper,
blows up to $\infty$ as $s\to 1-$, whereas $o(\epsilon^{2s})$ in (\ref{dwa}) is
uniform in the whole considered range $s\in
(\frac{1}{2},1)$. When $s\to 1$, the expansion (\ref{dwa}) becomes:
\begin{equation}\label{aed3}
\begin{split}
\frac{N+2}{N+\p}\cdot \fint_{B_\epsilon (x)}  u(y)\;\mathrm{d}y & + \frac{\p-2}{N+\p}\cdot\frac{1}{2}
\big(\sup_{B_\epsilon (x)}u +  \inf_{B_\epsilon (x)}u \big) \\ & = u(x) + \frac{\epsilon^2}{2(N+\p)}
|\nabla u(x)|^{2-\p}\Delta_\p u(x) + o(\epsilon^{2}),
\end{split}
\end{equation}
which in turn yields  (\ref{aed2}) for $\p=2$. We recall in passing that
(\ref{aed3}) is a convex combination of (\ref{aed2}) and the asymptotic
expansion for the infinity Laplacian in:
\begin{equation*}
\frac{1}{2} \big(\sup_{B_\epsilon (x)}u  + \inf_{B_\epsilon (x)}u \big)
= u(x) + \frac{\epsilon^2}{2}\Delta_\infty u(x) + o(\epsilon^2) \qquad \mbox{ as }\; \epsilon\to 0+, 
\end{equation*}
with the weights corresponding to the following identity for the classical
$\p$-Laplacian in non-divergence form: $|\nabla u|^{2-\p}\Delta_\p u = \Delta u + (\p-2)\Delta_\infty u$.

Asymptotic expansions for gradient-dependent operators have been
recently discussed in \cite{BS1, BS}. However, the averages in there depended on
$\nabla u(x)$, which is a drawback in the context of our further
applications, based on solutions to the truncated expansions
\ref{dppe}. We seek these solutions among the natural class of Borel 
functions. Indeed, they are at most continuous and become
higher regular only generically and in the limit as $\epsilon\to 0$,
so no notion of pointwise gradient may be feasible in the definition of an average. 

The idea of the local correction in the average $\bar {\mathcal{A}}_\epsilon$, and
of the expansion (\ref{raz}) with no reference to the gradient, first
appeared in \cite{EDL} in the context of the fractional
$\infty$-Laplacian $\Delta^s_\infty$. We also observe that the case
$\p=\infty$ where $\alpha=0$ in $T_\p^{0,
\infty}$, is independent and cannot be deduced from the present work. 
Expansion (\ref{aed3}) when $\Delta_\p u=0$, and the related
characterisation of $\p$-harmonic functions in the viscosity sense,
have been studied  in \cite{MPR}. This expansion informs a game-theoretical interpretation
of the $\p$-Laplacian (alternative to the one originally
carried out in \cite{PS}) only for $\p\geq 2$, when the weight coefficients
are nonnegative. Another expansion, yielding a family of Tug-of-War games in the
whole range $\p\in (1,\infty)$, was proposed in \cite{L1}. 

\smallskip

\subsection{Dynamic programming and Tug-of-War.} 
The second set of results in this paper concerns the operator $\mathcal{A}_\epsilon$
and the truncated version of the expansion (\ref{raz}), 
aiming at an approximation scheme for solutions to (\ref{ndci}). More
precisely, given an open bounded domain $\D\subset \R^N$ and a bounded
Borel data function $F:\R^N\setminus \D\to \R$, we consider the
following family of non-local averaging problems:
\begin{equation*}\tag*{${\mathrm{(DPP)}}_\epsilon$}
u_\epsilon(x) = \left\{\begin{array}{ll} \A_\epsilon u_\epsilon(x) &
    ~~\mbox{for } x\in\D\\ F(x) &  ~~\mbox{for } x\in\R^N\setminus \D.
\end{array}\right.
\end{equation*}
We prove that for every $\epsilon>0$ there exists exactly one
$u_\epsilon$ satisfying the above, which is bounded Borel on $\R^N$ (and continuous in $\D$).
We then show, for $\D$ satisfying the exterior cone condition
and for uniformly continuous $F$, that any sequence
$\{u_\epsilon\}_{\epsilon\to 0}$ has a further subsequence  converging
uniformly in $\R^N$ to a continuous limit $u$ that is a viscosity solution to (\ref{ndci}).
To this end, each $u_\epsilon(x)$ is shown to be the value of
the following zero-sum two-players game, which is a non-local version
of the Tug-of-War with noise introduced in \cite{PS}. 

In this game, each Player chooses a unit direction vector
according to their own strategy, based on the knowledge of all prior moves and
random outcomes.  With equal probabilities, direction from Player 1 or
Player 2 is picked; this resulting direction is called $y$. The
current game position $x_n$ is then updated to a next position 
$x_{n+1}$ within the shifted and truncated cone $x_n+T_\p^{\epsilon,\infty}(y)$,
randomly according to the probability-normalisation of the measure 
$|z|^{-N-2s}\mathrm{d}z$ on $T_\p^{\epsilon,\infty}(y)$. Such
process, started at a point $x_0\in\R^N$ is stopped the first time
when $x_n\not\in\D$, whereas Player 1 collects from their opponent the
payoff given by the value $F(x_n)$. We show that the expected value of the
payoff, under condition that both Players play optimally, has the
min-max property, yielding the solution $u_\epsilon$ to the
dynamic programming principle \ref{dppe}. 

Convergence as $\epsilon\to 0$ is obtained by showing the approximate
equicontinuity of the family $\{u_\epsilon\}_{\epsilon\to 0}$, for
which the sufficient condition is expressed via ``game-regularity'' of
the boundary points. This general condition, implied in particular 
by the exterior cone condition on $\partial\D$,
is similar in spirit to the celebrated Doob's boundary regularity criterion for Brownian motion.
The order of arguments follows then a general program in the context
of Tug-of-War games (see a recent textbook \cite{L}), which has been put forward in
\cite{PSSW} and which has so far yielded results for
$\p$-Laplacian, obstacle problems, subriemannian geometries and
time-dependent problems. The fact that this program can be carried out
in the present non-local setting, is not obvious, and it is another
main result of this work. 

\smallskip

\subsection{Outline of the paper.} 
We set the notation and introduce the main integral operators in
section \ref{sec2}. The non-local asymptotic expansions (\ref{raz}) and
(\ref{dwa}), together with precise bounds on their error terms,
are  proved in sections \ref{sec3} and \ref{sec4}, respectively. The
dynamic programming principles \ref{dppe} are discussed in
section \ref{sec_dppe}. The fact that the uniform
limits of their solutions $\{u_\epsilon\}_{\epsilon\to 0}$ are automatically
viscosity solutions to (\ref{ndci}), is shown in section
\ref{sec_visc}. The non-local Tug-of-War game is defined and proved to
yield solutions $u_\epsilon$ in section
\ref{sec_game}. Proofs of the asymptotic equicontinuity and game-regularity are carried out in sections
\ref{sec_ft} and \ref{sec_convergence}, where we rely on further
analysis of a barrier function from \cite{BCF2}. Finally, in Appendix
\ref{appendix} we prove uniqueness of viscosity solutions to
(\ref{ndci}) under a more restrictive assumption which necessitates
extending $\Delta_\p^s$ to include the case $\nabla u(x)=0$.
It is not clear if solutions to (\ref{ndci}) as posed originally, are unique.

\section{The fractional quotients and the fractional $\p$-Laplacian}\label{sec2}
 
We consider the following measure on the Borel subsets of $\R^N$:
$$\mathrm{d}\mu_s^N(z) \doteq \frac{C(N,s)}{|z|^{N+2s}}\;\mbox{d}z \quad \mbox{ where }\;
C(N,s)=\frac{4^s
  s\Gamma\big(\frac{N}{2}+s\big)}{\pi^{N/2}\Gamma\big(1-s\big)} =
\Big(\int_{\R^N}\frac{1-\cos \langle z, e_1\rangle}{|z|^{N+2s}}\;\mbox{d}z\Big)^{-1},$$
where the exponent $s$ in this paper is assumed to belong to the range:
$$s\in \big(\frac{1}{2},1\big).$$
One can show \cite{hitch} that $C(N,s) = s(1-s) c_{N,s}$ where $c_{N,s}$ is bounded and positive uniformly in $s$.
The role of the normalizing constant $C(N,s)$ is to ensure that the 
operator $-(-\Delta)^s$, given by: $-(-\Delta)^su(x) \doteq
-\int_{\R^N} u(x+z) + u(x-z) -2u(x) \dmN$, 
is a pseudo-differential operator with symbol $|\xi|^{2s}$. 

\begin{definition}\label{def_cone}
Fix $\p\in [2,\infty)$. We define the infinite cone $T_\p$ and the spherical cup $A_\p$:
$$T_\p\doteq \big\{z\in\R^N;~   \angle (e_1,z)<\alpha_\p \big\}, \qquad A_\p \doteq T_\p\cap \{|z|=1\},$$ 
where $\alpha_\p\in (0,\frac{\pi}{2}]$ is the angle such that:
\begin{equation}\label{Ap}
\p-1 = \frac{\int_{A_\p} \langle z,e_1\rangle^2
  \;\mbox{d}\sigma(z)}{\int_{A_\p} \langle z,e_2\rangle^2 \;\mbox{d}\sigma(z)}.
\end{equation}
For every $0\leq a \leq b\leq\infty$ and $|y|=1$,  we also have the
truncated cones:
$$T_\p^{a,b}(y)\doteq \big\{z\in\R^N;~ \angle (y,z)<\alpha_\p~~\mbox{
  and }~~ a<|z|<b\big\}, \qquad T_\p^{a,b}\doteq T_\p^{a,b}(e_1).$$  
Further, for two unit vectors $y\neq\tilde y$ we define the
rotation $R_{\tilde y, y}\in SO(N)$ as the unique orientation
preserving rotation, in plane spanned by $y, \tilde y$, and
such that $R_{\tilde y, y}y=\tilde y$. When $y=\tilde y$, we set: $R_{y,y}=Id_N$.
Note that: $T_\p^{a,b}(y) = R_{y,e_1} T_\p^{a,b}$ and $T_\p = T_\p^{0,\infty} =T_\p^{0,\infty}(e_1)$.
\end{definition}

The well posedness of this definition and its rationale will be explained
in Lemma \ref{lem_Tp}.

\smallskip

\begin{definition}\label{def_operators}
Fix an exponent $\p\in [2,\infty)$. Given a bounded, Borel function $u:\R^N\to\R$,
we define the following family of integral operators, parametrised by $\epsilon>0$:
\begin{equation*}
\Lse (x)  \doteq \sup_{|y|=1}\int_{T_\p^{\epsilon, \infty}(y)} u(x+z) - u(x) \dmN + 
\inf_{| y|=1} \int_{T_\p^{\epsilon, \infty}(y)} u(x+z) - u(x) \dmN.
\end{equation*}
When additionally $u\in C^{1,1}(x)$ and the corresponding gradient-like  vector $p_x\neq 0$, we have:
$$\Ls (x)  \doteq\int_{T_\p^{0,\infty}(\frac{p_x}{|p_x|})} \Lp(x,z,z)\dmN,$$
where for each $x, z, \tilde z\in\R^N$ we set:
$$\Lp(x, z, \tilde z) \doteq u(x+z) + u(x-\tilde z) - 2u(x).$$ 
\end{definition}

\smallskip

\begin{remark}
Recall that $u\in C^{1,1}(x)$ provided that there are $p_x\in \R^N$ and $C_x,r_x>0$ with:
\begin{equation}\label{C11}
\big|u(x+z) - u(x) - \langle p_x, z\rangle\big|\leq C_x |z|^2
\qquad\mbox{for all }\; |z|<r_x.
\end{equation}
One immediate consequence of (\ref{C11}) is that:
\begin{equation}\label{C2}
\big|\Lp (x,z,\tilde z) - \langle p_x, z-\tilde z\rangle \big|\leq C_x
\big(|z|^2 + |\tilde z|^2\big) \qquad\mbox{for all }\; |z|, |\tilde z|<r_x.
\end{equation}
Also, when $u\in C^2(\bar B_{r_x})$, where 
$B_{r_x}$ denotes the open ball centered at $x$ and with 
radius $r_x$, then condition (\ref{C11}) holds automatically with
$p_x = \nabla u(x)$ and $C_x = \frac{1}{2}\|\nabla^2u\|_{L^\infty(B_{r_x})}$. 
\end{remark}

\medskip

Since $\mu_s^N(T_\p^{\epsilon, \infty})<\infty$, each integral
$\int_{T_\p^{\epsilon, \infty}(y)}u(x+z)\dmN$ and consequently 
$\Lse (x)$ are well defined and finite for any bounded, Borel $u$. On the other hand,
$\mu_s^N(T_\p)=\infty$, so a corresponding formulation $\mathcal{L}_{s,\p}^0$
is in general not valid. We now observe:

\begin{proposition}\label{lem1}
Let $u:\R^N\to \R$ be a bounded, Borel function. Then: 
$$|\Lse (x)|\leq 2  C(N,s) |A_\p|\cdot \frac{\|u\|_{L^\infty}}{s \epsilon^{2s}} \quad\mbox{for all } x\in\R^N.$$
If moreover $u\in C^{1,1}(x)$, then $\Lse (x)$ are uniformly bounded in $\epsilon$:
\begin{equation*}
|\Lse (x)| \leq   C(N,s) |A_\p|\cdot \Big(\frac{C_x r_x^{2-2s}}{1-s} 
+ \frac{2\|u\|_{L^\infty}}{s r_x^{2s}}\Big).
\end{equation*}
When $p_x\neq 0$ then the same bound above holds for $\Ls(x)$.
\end{proposition}
\begin{proof}
The first claim is self-evident, because: $\mu_s^N(T_\p^{\epsilon,
  \infty})=  C(N,s) \int_\epsilon^\infty \frac{t^{N-1}
  |A_\p|}{t^{N+2s}}\;\mbox{d}t =  C(N,s)\frac{|A_\p|}{2s\epsilon^{2s}}$.
For the second claim, by changing variables we deduce that:
$$\Lse (x) = \sup_{|y|=1} \inf_{|\tilde y|=1} \int_{T_\p^{\epsilon,
    \infty}(y)} u(x+z) + u(x-R_{\tilde y, y}z) - 2u(x) \dmN.$$
Then, by (\ref{C2}) we get, for any $|y|=|\tilde y|=1$:
\begin{equation*}
\begin{split}
\Big|\int_{T_\p^{\epsilon, \infty}(y)}  u(x+z) + & \;u(x-R_{\tilde y, y}z) - 2u(x)  \dmN -
\int_{T_\p^{\epsilon, r_x}(y)} \langle p_x, z-R_{\tilde y, y}z\rangle
\dmN\Big| \\ & \leq \int_{T_\p^{\epsilon, r_x}(y)}
2C_x |z|^2\dmN + \int_{T_\p^{r_x, \infty}(y)}  4\|u\|_{L^\infty}\dmN
\\ & = C(N,s) |A_\p|\cdot \Big(\frac{C_x r_x^{2-2s}}{1-s}  + \frac{2\|u\|_{L^\infty}}{s r_x^{2s}}\Big).
\end{split}
\end{equation*}
On the other hand:
\begin{equation*}
\begin{split}
\sup_{|y|=1}\inf_{|\tilde y| =1} & \int_{T_\p^{\epsilon, r_x}(y)} \langle p_x, z-R_{\tilde y,
  y}z\rangle \dmN \\ & = \sup_{|y|=1}\inf_{|\tilde y| =1} \Big\langle p_x, 
\int_{T_\p^{\epsilon, r_x}(y)}z \dmN - \int_{T_\p^{\epsilon, r_x}(\tilde y)}z \dmN\Big\rangle = 0.
\end{split}
\end{equation*}
This results in:
\begin{equation*}
\begin{split}
& |\Lse (x) |  = \Big|\Lse (x) - \sup_{|y|=1}\inf_{|\tilde y| =1} \int_{T_\p^{\epsilon, r_x}(y)} \langle p_x, z-R_{\tilde y,
  y}z\rangle \dmN \Big | \\ & \leq
\sup_{|y|=|\tilde y|=1} \Big|\int_{T_\p^{\epsilon, \infty}(y)} u(x+z) + u(x-R_{\tilde y, y}z) - 2u(x) \dmN -
\int_{T_\p^{\epsilon, r_x}(y)} \langle p_x, z-R_{\tilde y, y}z\rangle \dmN\Big|,
\end{split}
\end{equation*}
ending the proof of the bound for $|\Lse (x) |$. The statement for $|\Ls (x)|$ follows similarly.
\end{proof}

We close this section by noting some useful identities:

\begin{lemma}\label{lem_Tp}
For every $\p\in [2, \infty)$ there exists $\alpha_\p\in (0, \frac{\pi}{2}]$ such that (\ref{Ap}) holds. 
Moreover:
\begin{enumerate}[leftmargin=7mm]
\item[(i)] $\displaystyle \int_{A_\p}\langle z,
  e_2\rangle^2\;\mathrm{d}\sigma(z)=\frac{|A_\p|}{N+\p-2}, ~$ 
and $~\displaystyle \int_{A_\p}\langle z,
  e_1\rangle^2\;\mathrm{d}\sigma(z)=\frac{\p-1}{N+\p-2} |A_\p|.$ 
We also have:  $~\displaystyle \int_{T_\p^{0,\epsilon}}\langle z,
  e_2\rangle^2\dmN = \frac{C(N,s) |A_\p|}{(N+\p-2) (2-2s)}\epsilon^{2-2s}$.

\item[(ii)] When $\nabla^2u(x)$ and $p_x\doteq \nabla u(x)\neq 0$ are well defined, then:
$$\int_{T_\p^{0,\epsilon}(\frac{p_x}{|p_x|})} \big\langle \nabla^2u(x) :
z\otimes z\big\rangle\dmN = \frac{C(N,s) |A_\p|}{(N+\p-2)(2-2s)}
\epsilon^{2-2s}\cdot |p_x|^{2-\p} \Delta_\p u(x).$$
\end{enumerate}
\end{lemma}
\begin{proof}
We consider the following function, which is continuous on $(0,\pi)$:
$$\alpha\mapsto Q(\alpha) \doteq \frac{\int_{A(\alpha)} \langle z,e_1\rangle^2
  \;\mbox{d}\sigma(z)}{\int_{A(\alpha)}\langle z,e_2\rangle^2
  \;\mbox{d}\sigma(z)}, \quad \mbox{ where }\; A(\alpha) = \{|z|=1;~\angle (e_1,z)<\alpha\}.$$ 
Since $Q(\frac{\pi}{2}) = {\frac{1}{2}\int_{\{|z|=1\} } \langle z,e_1\rangle^2
  \;\mbox{d}\sigma(z)}/\big({\frac{1}{2}\int_{\{|z|=1\} }\langle z,e_2\rangle^2 \;\mbox{d}\sigma(z)}\big)=1$,
while $\lim_{\alpha\to 0} Q(\alpha) = \infty$, it follows that for each $\p-1
\in [1,\infty)$ there indeed exists $\alpha_\p\in (0,\frac{\pi}{2}]$ satisfying $Q(\alpha_\p) = \p-1$.

To prove (i), we compute, putting $A_\p = A({\alpha_\p})$:
$$Q(\alpha_\p) = \frac{\int_{A_\p} 1\;\mbox{d}\sigma(z) - (N-1)\int_{A_\p}\langle z,
  e_2\rangle^2\;\mbox{d}\sigma(z)}{\int_{A_\p}\langle z,
  e_2\rangle^2\;\mbox{d}\sigma(z)} = \frac{1}{\fint_{A_\p}\langle z,
  e_2\rangle^2\;\mbox{d}\sigma(z)} - (N-1),$$
which implies that: $\fint_{A_\p}\langle z,
e_2\rangle^2\;\mbox{d}\sigma(z)=\frac{1}{N+\p-2}$, and consequently:  $\fint_{A_\p}\langle z,
e_1\rangle^2\;\mbox{d}\sigma(z)= 1- (N-1) \fint_{A_\p}\langle z,
e_2\rangle^2\;\mbox{d}\sigma(z)=\frac{\p-1}{N+\p-2}$.
On the other hand:
\begin{equation*}
\begin{split}
\int_{T_\p^{0,\epsilon}} \langle z, e_2\rangle^2\dmN & =
C(N,s)\int_0^\epsilon \frac{t^2 t^{N-1}}{t^{N+2s}} \int_{A_\p}\langle z,
  e_2\rangle^2\;\mbox{d}\sigma(z)\; \mbox{d}t \\ & = C(N,s) \int_{A_\p}\langle z,
  e_2\rangle^2\;\mbox{d}\sigma(z)\cdot \frac{\epsilon^{2-2s}}{2-2s}.
\end{split}
\end{equation*}

To prove (ii), observe that:
\begin{equation*}
\begin{split}
& \int_{T_\p^{0,\epsilon}} z\otimes z\dmN  = diag
\Big(\int_{T_\p^{0,\epsilon}} \langle z, e_1\rangle^2\dmN,
\int_{T_\p^{0,\epsilon}} \langle z, e_2\rangle^2\dmN, \ldots\Big) \\ & =
\int_{T_\p^{0,\epsilon}} \langle z, e_2\rangle^2\dmN\cdot Id_N +
\Big( \int_{T_\p^{0,\epsilon}} \langle z, e_1\rangle^2\dmN -
\int_{T_\p^{0,\epsilon}} \langle z, e_2\rangle^2\dmN\Big)
e_1\otimes e_1 \\ & = \int_{T_\p^{0,\epsilon}} \langle z,
e_2\rangle^2\dmN\cdot \Big( Id_N +(\p-2) \,e_1\otimes e_1\Big)
\\ & = \frac{C(N,s) |A_\p|}{(N+\p-2) (2-2s)}\epsilon^{2-2s}\cdot \Big( Id_N +(\p-2) \,e_1\otimes e_1\Big),
\end{split}
\end{equation*}
where we used:
$$\frac{ \int_{T_\p^{0,\epsilon}} \langle z, e_1\rangle^2\dmN}{
\int_{T_\p^{0,\epsilon}} \langle z, e_2\rangle^2\dmN}= Q(\alpha_\p)=\p-1.$$
It thus follows that:
\begin{equation*}
\begin{split}
& \Big\langle \nabla^2u(x) : \int_{T_\p^{0,\epsilon}(\frac{p_x}{|p_x|})} 
z\otimes z \dmN \Big\rangle \\ & = \frac{C(N,s) |A_\p|}{(N+\p-2) (2-2s)}\epsilon^{2-2s}
\Big\langle \nabla^2u(x) :  Id_N +(\p-2) \frac{p_x}{|p_x|}\otimes \frac{p_x}{|p_x|}\Big\rangle,
\end{split}
\end{equation*}
which completes the argument, upon recalling the formula: $\Delta_\p u =
|\nabla u|^{\p-2}\big(\Delta u + (\p-2)\Delta_\infty u)$ and the
definition of the $\infty$-Laplacian:
$\Delta_\infty u = \langle \nabla^2 u : \frac{\nabla u}{|\nabla u|}\otimes \frac{\nabla u}{|\nabla u|}\rangle$.
\end{proof}

\begin{remark}\label{rem2.5}
In \cite{BCF2}[Section 4.1.2],  the fractional $\p$-Laplacian
$\Delta_\p^su$ has been introduced by means of a scaled version 
of the operator $\Ls$. In particular, when $p_x\neq 0$ it follows that:
$$\Delta_\p^s u(x)= \frac{2-2s}{C(N,s)} \cdot\frac{N+\p-2}{|A_\p|}\Ls (x).$$
\end{remark}

\section{A non-local asymptotic expansion}\label{sec3}

In this section we prove the formula (\ref{raz}) with a precise form
of the error term. In what follows, we denote $B_r=B_r(x)$
for a fixed referential point $x\in\R^N$.
Given a function $u:\R^N\to\R$ that is uniformly continuous on 
$\R^N\setminus \bar B_r$, we denote its modulus of continuity by:
$$\omega_u (a)=\sup\big\{|u(z)-u(\bar z)|;~ z, \bar z \in \R^N\setminus
\bar B_{r_x},~ |z-\bar z|\leq a\big\}.$$

\begin{theorem}\label{thm4.7}
Let $u\in C^2(\bar B_{r_x})$ satisfy $p_x\doteq \nabla
u(x)\neq 0$, and denote $C_x\doteq \frac{1}{2}\|\nabla^2u\|_{L^\infty(B_{r_x})}$.
Assume that $u$ is uniformly continuous on
$\R^N\setminus \bar B_{r_x}$ with modulus of continuity $\omega_u$. Recall that:
$$\A_\epsilon u(x) \doteq
\frac{1}{2}\Big(\sup_{|y|=1}\fint_{T_\p^{\epsilon, \infty}(y)} u(x+z)
\dmN + \inf_{|y|=1}\fint_{T_\p^{\epsilon, \infty}(y)} u(x+z)\dmN \Big).$$
Then there holds:
\begin{equation}\label{sec3_m}
\begin{split}
\Big|\A_\epsilon&u(x) - u(x) -
\frac{s}{C(N,s) |A_\p|}\epsilon^{2s}\Ls (x) \Big| \\ & \leq  \frac{s}{1-s}\cdot
C_x \epsilon^2 + \epsilon^{2s}\Big(4
sC_x\frac{r_x^{2-2s}-\epsilon^{2-2s}}{1-s}\cdot m_\epsilon
+ \big(r_x^{-2s} + \frac{2s}{2s-1}r_x^{1-2s}\big)\cdot \omega_u(m_\epsilon)\Big),
\end{split}
\end{equation}
where we define:
\begin{equation}\label{meke}
\begin{split}
&  m_\epsilon = \max\Big\{
\frac{16 (N+\p-2)}{\p-1} \cdot \frac{C_x}{|p_x|} \cdot\frac{2s-1}{1-s}\cdot 
\frac{r_x^{2-2s}-\epsilon^{2-2s}}{\epsilon^{1-2s} -{r_x}^{1-2s}}, ~  ~\kappa_\epsilon\Big\}, \\ &
\kappa_\epsilon = \sup\Big\{ m;~ m\in [0,2] \; \mbox{ and } \; m^2
\leq  \frac{N+\p-2}{\p-1}\cdot\frac{8\omega_u(m)}{|p_x|} \cdot \frac{\frac{2s-1}{2s} 
  r_x^{-2s} +r_x^{1-2s}}{\epsilon^{1-2s} - r_x^{1-2s}}\Big\}.
\end{split}
\end{equation}
\end{theorem}

\medskip

Before giving the proof, a few observations are in order:

\begin{remark}\label{rem4}
\begin{enumerate}[leftmargin=7mm]
\item [(i)] Using the crude bound $\omega_u\leq 2\|u\|_{L^\infty}$, it
  follows that for all $\epsilon<\frac{r_x}{2}$ we have:
\begin{equation*}
\begin{split}
\kappa_\epsilon & \leq
\Big(\frac{16\|u\|_{L^\infty}}{|p_x|}\cdot \frac{N+\p-2}{\p-1} \cdot
\frac{\frac{2s-1}{2s} r_x^{-2s} + r_x^{1-2s}}{\epsilon^{1-2s} - r_x^{1-2s}}\Big)^{1/2}
\\ & \leq 8 \Big(\frac{\|u\|_{L^\infty}}{|p_x|}\cdot \frac{N+\p-2}{\p-1}\cdot
\frac{r_x^{-2s} + r_x^{1-2s}}{2s-1}\Big)^{1/2}\epsilon^{s-1/2}.
\end{split}
\end{equation*}
In the second inequality above
we used that $\epsilon<\frac{r_x}{2}$ implies: $\epsilon^{1-2s} -
r_x^{1-2s}> \epsilon^{1-2s} (1-2^{1-2s})$, and that $1-2^{1-2s}\geq (2s-1)\ln\sqrt{2}
>\frac{2s-1}{4}$ in the range $s\in (\frac{1}{2},1)$.
Since the first quantity in $m_\epsilon$ is of order $\epsilon^{2s-1}$, the
bounding coefficient of order $\epsilon^{2s}$ in (\ref{sec3_m}) becomes: 
$$ C_{N,\p, s} C(r_x)\cdot C\big(\frac{\|u\|_{L^\infty}}{|p_x|}\big)\cdot 
\big( C_x \epsilon^{s-1/2} + \omega_u(\epsilon^{s-1/2})\big), $$
where $C(\alpha)$ depends only on the indicated quantity $\alpha$.
\item[(ii)] When $u\in C^{0,\alpha}(\R^N\setminus \bar B_{r_x})$ with
$\alpha\in (0,1)$, then $\omega_u(m)=[u]_\alpha m^{\alpha}$ and
we similarly obtain:
$$\kappa_\epsilon \leq
\Big(\frac{32\;[u]_{\alpha}}{|p_x|}\cdot \frac{N+\p-2}{\p-1} \cdot
\frac{r_x^{-2s} +r_x^{1-2s}}{2s-1}\Big)^{\frac{1}{2-\alpha}}\epsilon^{\frac{2s-1}{2-\alpha}},$$
resulting in the following bounding constant:
$$ C_{N,\p, s} C(r_x)\cdot C\big(\frac{[u]_{\alpha}}{|p_x|}\big)\epsilon^{\alpha\cdot\frac{2s-1}{2-\alpha}}.$$
\item[(iii)] When $u$ is  Lipschitz on $\R^N\setminus \bar
  B_{r_x}$ with the Lipschitz constant $\mbox{Lip}_u$, we get: 
$$\kappa_\epsilon \leq \frac{32 \; \mbox{Lip}_u}{|p_x|}\cdot \frac{N+\p-2}{\p-1} \cdot
\frac{r_x^{-2s} + r_x^{1-2s}}{2s-1}\epsilon^{2s-1},$$
so both quantities in $m_\epsilon$ are of the same order $\epsilon^{2s-1}$ and:
$$ m_\epsilon\leq \frac{32}{|p_x|}\cdot \frac{N+\p-2}{\p-1} \cdot\max\Big\{\frac{2C_x r_x^{2-2s}}{1-s},
\frac{\mbox{Lip}_u\cdot (r_x^{-2s}+r_x^{1-2s})}{2s-1}\Big\}\epsilon^{2s-1}.$$
Consequently, the discussed bounding expression becomes:
$$ C_{N,\p,s} C(r_x)\cdot C\big(C_x,\frac{1}{|p_x|}, \mbox{Lip}_u\big)\epsilon^{2s-1}. $$
\end{enumerate}
\end{remark}

\medskip

In order to estimate the difference of $\Lse$ and $\Ls$ when $p_x\neq 0$, we will
analyze the behaviour of approximations to the extremizers $y, \tilde y$ in Definition \ref{def_operators}.
The proof below follows the outline of the proof of \cite[Theorem
1]{EDL} in the fractional $\infty$-Laplacian setting.

\begin{lemma}\label{prop3}
Under the same assumptions and notation as in Theorem \ref{thm4.7}, 
for every $\epsilon< r_x$ there holds, with  the quantity $m_\epsilon$ is as (\ref{meke}).
\begin{equation}\label{tre}
\begin{split}
\Big|\Lse (x)- & \int_{T_\p^{\epsilon,
    \infty}(\frac{p_x}{|p_x|})}\Lp(x, z, z)\dmN\Big|\\ & \leq C(N,s)|A_\p|\cdot\Big(
4m_\epsilon C_x\frac{r_x^{2-2s}-\epsilon^{2-2s}}{1-s} + 
2\omega_u(m_\epsilon)\Big( \frac{r_x^{-2s}}{2s} + \frac{r_x^{1-2s}}{2s-1}\Big)\Big),
\end{split}
\end{equation}
\end{lemma}
\begin{proof}
{\bf 1.} For every $\epsilon<\eta_x$ and every $\delta>0$ satisfying:
\begin{equation}\label{tre.5}
\delta\leq \frac{\big(16 C_x \int_{T_\p^{\epsilon,
      r_x}}|z|^2\dmN\big)^2}{|p_x|\int_{T_\p^{\epsilon, r_x}}\langle z,e_1\rangle\dmN},
\end{equation}
let $|y_\delta^\epsilon|=1$ be such that
$\displaystyle \sup_{|y|=1}\int_{T_\p^{\epsilon, \infty}(y)}u(x+z)\dmN
\leq \int_{T_\p^{\epsilon, \infty}(y_\delta^\epsilon)}u(x+z)\dmN
+\delta$. Then:
\begin{equation*}
\begin{split}
\delta &\geq \int_{T_\p^{\epsilon,
    \infty}(\frac{p_x}{|p_x|})}u(x+z)\dmN - \int_{T_\p^{\epsilon, \infty}(y_\delta^\epsilon)}u(x+z)\dmN 
\\ & \geq \int_{T_\p^{\epsilon, r_x}(\frac{p_x}{|p_x|})}u(x+z)-
u(x+R_{y_\delta^\epsilon, \frac{p_x}{|p_x|}}z)\dmN -
\omega_u \big (|Id_N - R_{y_\delta^\epsilon, \frac{p_x}{|p_x|}}|\big) \cdot \int_{T_\p^{r_x, \infty}}(1+|z|)\dmN, 
\end{split}
\end{equation*}
because:
\begin{equation*}
\begin{split}
&\int_{T_\p^{r_x, \infty}(\frac{p_x}{|p_x|})} \big|u(x+z) - u(x+
R_{y_\delta^\epsilon, \frac{p_x}{|p_x|}}z)\big| \dmN \\ & \quad \leq 
\int_{T_\p^{r_x, \infty}(\frac{p_x}{|p_x|})} (1+|z|) \cdot
\sup\Big\{|u(y_1) - u(y_2)|; ~ y_1, y_2\not\in\bar B_{r_x}, ~
|y_1-y_2|\leq |Id_N - R_{y_\delta^\epsilon, \frac{p_x}{|p_x|}}|\Big\} \dmN
\\ & \quad \leq \omega_u(|Id_N - R_{y_\delta^\epsilon, \frac{p_x}{|p_x|}}|)\cdot 
\int_{T_\p^{r_x, \infty}} (1+|z|)  \dmN.
\end{split}
\end{equation*}
Call $m\doteq \big|\frac{p_x}{|p_x|}-y_\delta^\epsilon\big|$ and note that
$|Id_N - R_{y_\delta^\epsilon, \frac{p_x}{|p_x|}}|=m$. Recalling (\ref{C11}) further leads to:
\begin{equation}\label{tre.6}
\begin{split}
\delta &\geq 
\int_{T_\p^{\epsilon, r_x}(\frac{p_x}{|p_x|})}\Big\langle\nabla u(x+z), 
(Id_N - R_{y_\delta^\epsilon, \frac{p_x}{|p_x|}})z\Big\rangle\dmN \\ &\quad
- C_x \int_{T_\p^{\epsilon, r_x}} m^2 |z|^2 \dmN - 
\omega_u(m) \int_{T_\p^{r_x, \infty}} (1+|z|)  \dmN 
\\ & \geq \Big\langle p_x, (Id_N - R_{y_\delta^\epsilon, \frac{p_x}{|p_x|}})
\int_{T_\p^{\epsilon, r_x}(\frac{p_x}{|p_x|})} z  \dmN\Big\rangle  \\ & \quad
- C_x m (2+m) \cdot \int_{T_\p^{\epsilon, r_x}} |z|^2  \dmN -
\omega_u(m)\cdot \int_{T_\p^{r_x, \infty}} (1+|z|)  \dmN,
\end{split}
\end{equation}
since $\big|\nabla u(x+z) - \nabla u(x)\big|\leq 2C_x|z|$ for all $z\in B_{r_x}$.

We now observe that:
$ \int_{T_\p^{\epsilon, r_x}(\frac{p_x}{|p_x|})} z  \dmN=
\int_{T_\p^{\epsilon, r_x}} \langle z, e_1\rangle  \dmN\cdot
\frac{p_x}{|p_x|}$, which yields that the first term in the right
hand side of (\ref{tre.6}) equals:
$$\int_{T_\p^{\epsilon, r_x}} \langle z, e_1\rangle \dmN\cdot
\Big\langle p_x, (Id_N - R_{y_\delta^\epsilon, \frac{p_x}{|p_x|}})
\frac{p_x}{|p_x|} \Big\rangle  = \int_{T_\p^{\epsilon, r_x}} \langle z, e_1\rangle  \dmN\cdot
\big\langle p_x, \frac{p_x}{|p_x|} - y_\delta^\epsilon\big\rangle.$$
Finally, noting the identity:
\begin{equation*}
\begin{split}
m^2=\Big|\frac{p_x}{|p_x|} - y_\delta^\epsilon\Big|^2 = 2 - 2 \Big\langle
\frac{p_x}{|p_x|},  y_\delta^\epsilon\Big\rangle = \frac{2}{|p_x|}
\Big\langle p_x,\frac{p_x}{|p_x|} - y_\delta^\epsilon\Big\rangle,
\end{split}
\end{equation*}
the discussed term becomes:
$$\int_{T_\p^{\epsilon, r_x}} \langle z, e_1\rangle  \dmN \cdot \frac{m^2|p_x|}{2}.$$

Consequently, it follows by (\ref{tre.6}) that:
\begin{equation}\label{quattro}
\begin{split}
m^2 \leq  2\cdot \frac{\delta+ 4C_xm
\int_{T_\p^{\epsilon, r_x}} |z|^2  \dmN +\omega_u(m)
\int_{T_\p^{r_x, \infty}} (1+ |z|) \dmN }{|p_x|\int_{T_\p^{\epsilon, r_x}} \langle z, e_1\rangle  \dmN}.
\end{split}
\end{equation}

\smallskip

{\bf 2.} We now analyze the bound (\ref{quattro}) in the following distinct cases. In
the first case:
\begin{equation}\label{cinque}
m^2\leq \frac{4\delta}{|p_x|\int_{T_\p^{\epsilon, r_x}} \langle z, e_1\rangle  \dmN}\leq
\Big(\frac{32 C_x \int_{T_\p^{\epsilon, r_x}} |z|^2 \dmN}{|p_x|\int_{T_\p^{\epsilon, r_x}} \langle z, e_1\rangle  \dmN}\Big)^2,
\end{equation}
where we used (\ref{tre.5}) in the second inequality. In the reverse
case, we 
get:
\begin{equation*}
m^2 \leq  \frac{16C_xm
\int_{T_\p^{\epsilon, r_x}}|z|^2\dmN +4\omega_u(m) \int_{T_\p^{r_x, \infty}}
(1+|z|)\dmN}{|p_x| \int_{T_\p^{\epsilon, r_x}}\langle z, e_1\rangle \dmN}\doteq I_1 + I_2.
\end{equation*}

When $I_2\leq I_1$, then the above yields the same bound as in (\ref{cinque}), namely:
\begin{equation*}
m\leq \frac{32 C_x \int_{T_\p^{\epsilon, r_x}} |z|^2
  \dmN}{|p_x|\int_{T_\p^{\epsilon, r_x}} \langle z, e_1\rangle  \dmN}
=  \frac{16 C_x}{|p_x|\fint_{A_\p}\langle z, e_1\rangle \;\mbox{d}\sigma(z)}\cdot\frac{2s-1}{1-s}\cdot 
\frac{r_x^{2-2s}-\epsilon^{2-2s}}{\epsilon^{1-2s} -{r_x}^{1-2s}}. 
\end{equation*}
When $I_1<I_2$, then:
\begin{equation*}
m^2 \leq  \frac{8\omega_u(m) \int_{T_\p^{r_x, \infty}}
(1+|z|)\dmN}{|p_x| \int_{T_\p^{\epsilon, r_x}}\langle z, e_1\rangle \dmN}
= \frac{8\omega_u(m)}{|p_x|\fint_{A_\p}\langle z, e_1\rangle\;\mbox{d}\sigma(z)}\cdot \frac{\frac{2s-1}{2s}
  r_x^{-2s} +r_x^{1-2s}}{\epsilon^{1-2s} - r_x^{1-2s}},
\end{equation*}
so that $ m\leq \kappa_\epsilon$ as $\fint_{A_\p}\langle z, e_1\rangle\;\mbox{d}\sigma(z)\geq
\fint_{A_\p}\langle z, e_1\rangle^2\;\mbox{d}\sigma(z)=\frac{\p-1}{N+\p-2}.$ 
Hence $m\leq m_\epsilon$ in both cases.

\smallskip

{\bf 3.} By the same analysis as in step 1, we see that the unit
vector  $\tilde y_\delta^\epsilon$ satisfying:
$$\inf_{|y|=1}\int_{T_\p^{\epsilon, \infty}(y)} u(x+z) - u(x)\dmN \geq
\int_{T_\p^{\epsilon, \infty}(-\tilde y_\delta^\epsilon)} u(x+z) - u(x) \dmN -\delta,$$
differs from the unit vector $\frac{p_x}{|p_x|}$ at most by $m_\epsilon$. Note that:
\begin{equation*}
\begin{split}
&\int_{T_\p^{\epsilon, \infty}(\tilde y_\delta^\epsilon)} L_u(x,z,z) \dmN -\delta
\\ & = \int_{T_\p^{\epsilon, \infty}(\tilde y_\delta^\epsilon)} u(x+z) -
u(x) \dmN + \int_{T_\p^{\epsilon, \infty}(-\tilde y_\delta^\epsilon)} u(x+z) - u(x) \dmN 
-\delta \leq \Lse (x) \\ & \leq \int_{T_\p^{\epsilon, \infty}( y_\delta^\epsilon)} u(x+z) - u(x) \dmN + \int_{T_\p^{\epsilon,
\infty}(-y_\delta^\epsilon)} u(x+z) - u(x) \dmN  +\delta \\ & =
\int_{T_\p^{\epsilon, \infty}( y_\delta^\epsilon)} L_u(x,z,z) \dmN +\delta,
\end{split}
\end{equation*}
which yields:
\begin{eqnarray}\label{sette}
\begin{split}
& \Big|\Lse (x) - \int_{T_\p^{\epsilon, \infty}(\frac{p_x}{|p_x|})} L_u(x,z,z) \dmN 
\Big|\leq \delta + \max \big\{ |J(y_\delta^\epsilon)|, |J(\tilde y_\delta^\epsilon|)\big\},\\ &
\mbox{where: } J(y)\doteq \int_{T_\p^{\epsilon,
    \infty}(\frac{p_x}{|p_x|})} u(x+z) - u(x+R_{y,\frac{p_x}{|p_x|}}z) + u(x-z) - u(x - R_{y,\frac{p_x}{|p_x|}}z) \dmN.
\end{split}
\end{eqnarray}

We now estimate the two terms $ J(y_\delta^\epsilon)$, $ J(\tilde
y_\delta^\epsilon)$ and show that they are bounded independently of
$\delta$. This will allow to pass $\delta\to 0$ in (\ref{sette}) and directly
conclude the claimed estimate (\ref{tre}).
We start by splitting the integral in $J(y_\delta^\epsilon)$ in two
terms: $|J(y_\delta^\epsilon)|\leq J_1+ J_2$, where:
\begin{equation*}
\begin{split}
J_1 & = \Big|\int_{T_\p^{\epsilon, r_x}(\frac{p_x}{|p_x|})} u(x+z) -
u(x+R_{y_\delta^\epsilon,\frac{p_x}{|p_x|}}z) + u(x-z) - u(x -
R_{y_\delta^\epsilon,\frac{p_x}{|p_x|}}z) \dmN\Big|
\\ & \leq \int_{T_\p^{\epsilon, r_x}(\frac{p_x}{|p_x|})} \Big|\big\langle
\nabla u(x+z), z-R_{y_\delta^\epsilon, \frac{p_x}{|p_x|}}z\big\rangle
- \big\langle \nabla u(x-z), z-R_{y_\delta^\epsilon,
  \frac{p_x}{|p_x|}}z\big\rangle\Big|\dmN \\ & \quad + \int_{T_\p^{\epsilon,
    r_x}} 2C_x m^2|z|^2\dmN \\ & \leq (4m+2m^2)\int_{T_\p^{\epsilon, r_x}}
C_x |z|^2\dmN\leq 8mC_x \int_{T_\p^{\epsilon, r_x}} |z|^2\dmN.
\end{split}
\end{equation*}
The remaining estimate is:
\begin{equation*}
\begin{split}
J_2 & = \Big|\int_{T_\p^{r_x, \infty}(\frac{p_x}{|p_x|})} u(x+z) -
u(x+R_{y_\delta^\epsilon,\frac{p_x}{|p_x|}}z) + u(x-z) - u(x - R_{y_\delta^\epsilon,\frac{p_x}{|p_x|}}z) \dmN\Big| 
\\ & \leq 2\omega_u(m) \int_{T_\p^{ r_x, \infty}}(1+|z|)\dmN 
\end{split}
\end{equation*}

In conclusion, we obtain:
$$ |J(y_\delta^\epsilon)| \leq 8m_\epsilon C_x \int_{T_\p^{\epsilon, r_x}}
|z|^2\dmN + 2\omega_u(m_\epsilon) \int_{T_\p^{ r_x, \infty}}(1+|z|)\dmN.$$
Clearly, $|J(\tilde y_\delta^\epsilon)| $ enjoys the same
bound. The result in Lemma now follows by (\ref{sette}).
\end{proof}

\medskip

By Remark \ref{rem4} (iii) we see that in case of $u$ which is
Lipschitz on $\R^N\setminus \bar B_{r_x}$, the order of the error
bounding quantity in Theorem \ref{thm4.7} is
$C(s)\cdot(\epsilon^{4s-1}+\epsilon^2)$ as $\epsilon\to 0+$, where
$C(s)$ blows up as $s\to 1-$. This drawback will be remedied by means
of another asymptotic expansion in Theorem \ref{thm6}, proved in
section \ref{sec4}. We are now ready to give:

\noindent {\bf Proof of Theorem \ref{thm4.7}}

In view of (\ref{C2}) we get:
\begin{equation*}
\begin{split}
\Big|\Ls (x) - \int_{T_\p^{\epsilon, \infty}(\frac{p_x}{|p_x|})} \Lp(x,z,z)\dmN\Big|
& \leq \int_{T_\p^{0,\epsilon}(\frac{p_x}{|p_x|})} |\Lp(x,z,z)|\dmN \\
& \leq C(N,s) |A_\p| \cdot C_x\frac{\epsilon^{2-2s}}{1-s}.
\end{split}
\end{equation*}
Consequently, Lemma \ref{prop3} yields:
\begin{equation*}
\begin{split}
\Big|\Lse (x&) - \Ls (x)\Big| \leq \frac{4 C(N,s) |A_\p|}{1-s} \cdot C_x \big(r_x^{2-2s} -
  \epsilon^{2-2s}\big) \cdot m_\epsilon \\ & + \frac{C(N,s) |A_\p|}{s}\cdot \Big(r_x^{-2s} +
\frac{2s}{2s-1}r_x^{1-2s}\Big)\cdot \omega_u(m_\epsilon) +
\frac{C(N,s) |A_\p|}{1-s} \cdot C_x\epsilon^{2-2s}.
\end{split}
\end{equation*}
The result follows by collecting terms and scaling by the factor: $\frac{s}{C(N,s)|A_\p|}\epsilon^{2s}$.
\endproof

\section{A local - non-local asymptotic expansion}\label{sec4}

In this section we present a refined version of the argument in
Theorem \ref{thm4.7}. We need one more estimate before giving the
proof of expansion (\ref{dwa}) in Theorem \ref{thm6}.

\begin{proposition}\label{prop5}
Let $u\in C^2(\bar B_{r_x})$ satisfy: $p_x\doteq \nabla
u(x)\neq 0$. For every $\epsilon< r_x$ such that $\epsilon
|\nabla^2u(x)|\leq |p_x|$, denote $B_\epsilon = B_\epsilon(x)$. Then there holds:
\begin{equation*}
\begin{split}
\Big| &\frac{C(N,s) |A_\p|}{(N+p-2)(1-s)}\cdot \epsilon^{-2s}\Big(\frac{\p-2}{2}
\big(\sup_{B_\epsilon} u + \inf_{B_\epsilon} u \big) +
(N+2)\fint_{B_\epsilon} u(y)\;\mathrm{d}y - (N+\p) u(x)\Big)
\\ & \qquad - \int_{T_\p^{0,\epsilon}(\frac{p_x}{|p_x|})} L_u(x,z,z)\dmN\Big| \leq 
\frac{C(N,s)|A_\p|}{1-s}\cdot\epsilon^{2-2s}
\sup_{y\in B_\epsilon}|\nabla^2u(y) - \nabla^2u(x)| \\ &
\qquad\qquad\qquad \qquad\qquad\qquad \qquad\qquad \quad + \frac{2C(N,s) |A_\p|}{1-s}
\cdot\frac{\p-2}{N+\p-2} \cdot \epsilon^{3-2s} \frac{|\nabla^2u(x)|^2}{|p_x|}.
\end{split}
\end{equation*}
\end{proposition}
\begin{proof}
An application of Taylor's expansion and the identity in Lemma \ref{lem_Tp} (iii)
results in the following estimate:
\begin{equation*}
\begin{split}
\Big| & \int_{T_\p^{0,\epsilon}(\frac{p_x}{|p_x|})} L_u(x,z,z)\dmN
- \frac{C(N,s) |A_\p|}{(N+\p-2) (2-2s)}\cdot \epsilon^{2-2s}|p_x|^{2-\p}\Delta_\p u(x)\Big| \\ & \leq 
\int_{T_\p^{0,\epsilon}(\frac{p_x}{|p_x|})} |L_u(x,z,z) - \langle
\nabla^2 u(x) : z\otimes z\rangle|\dmN \\ & \leq \int_{T_\p^{0,\epsilon}}
|z|^2\dmN\cdot \sup_{y\in B_\epsilon}|\nabla^2u(y) - \nabla^2u(x)| 
= \frac{C(N,s)|A_\p|}{2-2s}\cdot \epsilon^{2-2s}\sup_{y\in B_\epsilon}|\nabla^2u(y) - \nabla^2u(x)|.
\end{split}
\end{equation*}
We now invoke the following folklore claim (whose proof we recall below):
\begin{equation}\label{otto}
\begin{split}
\Big|&(\p-2)\big( \sup_{B_\epsilon} u + \inf_{B_\epsilon} u\big) +2(N+2)\fint_{B_\epsilon}u(y)\;\mathrm{d}y
-2(N+\p)u(x) - \epsilon^2|p_x|^{2-\p}\Delta_\p u(x)\Big| \\ & \leq
4\epsilon^3(\p-2) \frac{|\nabla^2u(x)|^2}{|p_x|}
+ \epsilon^2 (N+\p-2)\cdot \sup_{y\in B_\epsilon}|\nabla^2u(y) - \nabla^2u(x)|.
\end{split}
\end{equation}
Summing up the scaled versions of the last two displayed formulas completes the argument.
\end{proof}

\medskip

\noindent {\bf Proof of claim (\ref{otto}).}

{\bf 1.} We first deduce the bound in the particular case when
$\epsilon=1$, $x=0$ and $u$ is a quadratic polynomial
with gradient given by a unit vector $p\in\R^N$ and Hessian given by a
symmetric matrix $B\in \R^{N\times N}$ satisfying $|B|\leq 1$:
$$u(z)=\langle p, z\rangle + \frac{1}{2}\langle B: z\otimes z\rangle.$$
It is straightforward that:
\begin{equation}\label{one}
\fint_{B_1} u(z)\;\mbox{d}z=\frac{1}{2}\Big\langle
B:\fint_{B_1}z\otimes z\;\mbox{d}z \Big\rangle = \frac{1}{2(N+2)}\big\langle Id_N:B\big\rangle
= \frac{1}{2(N+2)} \Delta u(0).
\end{equation}
In order to address the nonlinear averaging quantity $\sup_{B_1} u +
\inf_{B_1} u$, consider $z_{max}\in \bar B_1$ that is a maximizer of $u$. Note that 
$|z_{max}|=1$, because $|\nabla u(z)|= |p+Bz|> 1 -|B|\geq 0$
for all $|z|\leq 1$. Further, since $u(z_{max})\geq u(p)$, there holds:
$$ \langle p, z_{max}\rangle\geq 1+ \frac{1}{2}\langle B: p\otimes p -
z_{max}\otimes z_{max}\rangle \geq 1 - |z_{max}-p|\cdot |B|.$$
Consequently: $|z_{max}-p|^2 = 2- 2\langle z_{max}, p\rangle \leq 2 |z_{max}-p|\cdot |B|$, so that:
$$|z_{max}-p|\leq 2|B|.$$
Noting that $\langle p, z_{max}-p\rangle\leq 0$, we hence arrive at:
$$0\leq u(z_{max})-u(p) = \langle p, z_{max}-p\rangle +\frac{1}{2}\langle B:
z_{max}\otimes z_{max} - p\otimes p \rangle\leq |z_{max}-p|\cdot
|B|\leq 2 |B|^2.$$

An entirely similar argument yields the bound for a minimizer
$z_{min}$ of $u$ on $\bar B_1$:
$$0\geq u(z_{min})-u(-p) \geq -2 |B|^2,$$
which results in the bound:
\begin{equation}\label{two}
\begin{split}
\big| \sup_{B_1} u + \inf_{B_1} u - \Delta_\infty u(0)\big| & =
\big|u(z_{max})+u(z_{min}) - \langle B: p\otimes p\rangle\big|
\\ & = \big|u(z_{max}) - u(p) + u(z_{min}) - u(-p)\big| \leq 4|B|^2.
\end{split}
\end{equation}
The estimate (\ref{otto}) follows in the present case summing
(\ref{one}) scaled by $2(N+2)$ and (\ref{two}) scaled by $\p-2$.

\smallskip

{\bf 2.} For the general case, we may still assume $x=0$ and
$u(0)=0$. Define $\bar u(y)= \langle p_x, y\rangle +
\frac{1}{2}\langle \nabla^2 u(0):y\otimes y\rangle$ and note that:
\begin{equation}\label{nove}
\begin{split}
 \frac{2(N+2)}{N}\cdot\Big |\fint_{B_\epsilon} u(y)\;\mbox{d}y - \fint_{B_\epsilon} & \bar
u(y)\;\mbox{d}y\Big|,\quad 
\Big|\big(\sup_{B_\epsilon} u + \inf_{B_\epsilon} u\big) -
\big(\sup_{B_\epsilon}\bar u + \inf_{B_\epsilon}\bar u\big) \Big|
\\ & \leq \epsilon^2\sup_{y\in B_\epsilon}|\nabla^2 u(y) -
\nabla^2 u(0)|,
\end{split}
\end{equation}
by means of Taylor's expansion. We now apply the conclusion of step 1 to the rescaled function:
$$\frac{1}{\epsilon |p_x|} \bar u(\epsilon z)=
\big\langle\frac{p_x}{|p_x|}, z\big\rangle + \frac{1}{2}\Big\langle \epsilon
\frac{\nabla^2 u(0)}{|p_x|}:z\otimes z\Big\rangle\quad \mbox{ for
}\; z\in \bar B_1.$$ 
It follows that for all $\epsilon$ satisfying $\epsilon |\nabla^2 u(0)|\leq |p_x|$, we get:
\begin{equation*}
\begin{split}
\Big|& \frac{\p-2}{\epsilon |p_x|}\big(\sup_{B_\epsilon}\bar u +
\inf_{B_\epsilon}\bar u\big) + \frac{2(N+2)}{\epsilon |p_x|}\fint_{B_\epsilon}
\bar u(y)\;\mbox{d}y
- \Big( \epsilon \frac{\Delta u(0)}{|p_x|} + (\p-2)\Big\langle
\epsilon\frac{\nabla^2 u(0)}{|p_x|}:\frac{p_x}{|p_x|}
\otimes \frac{p_x}{|p_x|}\Big\rangle \Big)\Big| \\ & \leq
4\epsilon^2\cdot (\p-2) \frac{|\nabla^2 u(0)|^2}{|p_x|^2},
\end{split}
\end{equation*}
which yields:
$$\big|(\p-2)\big(\sup_{B_\epsilon}\bar u +
\inf_{B_\epsilon}\bar u\big) +2(N+2) \fint_{B_\epsilon}\bar
u(y)\;\mbox{d}y - \epsilon^2 |p_x|^{2-\p}\Delta_\p u(0) \big| \leq
4\epsilon^3 (\p-2)\frac{|\nabla^2u(0)|^2}{|p_x|}.$$
Combined together with (\ref{nove}), the above estimate ends the proof of (\ref{otto}).
\endproof

\medskip

\begin{theorem}\label{thm6}
Under the same assumptions and notation as in Theorem \ref{thm4.7},
for every $\epsilon< r_x$ such that $\epsilon
|\nabla^2\phi(x)|\leq |p_x|$,  denote $B_\epsilon=B_\epsilon(x)$ and 
consider the average:
\begin{equation*}
\begin{split}
\bar\A_\epsilon u(x) = &\; \frac{(1-s)(N+\p-2)}{N+\p-2+2s}\cdot\frac{1}{2}\Big(\sup_{|y|=1}
\fint_{T_\p^{\epsilon, \infty}(y)} u(x+z)\dmN + \inf_{|y|=1}\fint_{T_\p^{\epsilon, \infty}(y)}
u(x+z)\dmN \Big) \\ & + \frac{s(\p-2)}{N+\p-2+2s}\cdot\frac{1}{2}
\big(\sup_{B_\epsilon}u + \inf_{B_\epsilon}u \big) + \frac{s(N+2)}{N+\p-2+2s}\fint_{B_\epsilon}
u(y)\;\mathrm{d}y.
\end{split}
\end{equation*}
Then there holds:
\begin{equation*}
\begin{split}
\Big|& \bar\A_\epsilon u(x) - u(x) - \frac{(1-s)s}{C(N,s)
  |A_\p|}\cdot\frac{N+\p-2}{N+\p-2+2s}\cdot \epsilon^{2s} \Ls (x) \Big|\\ & \leq 
\frac{4s (N+\p-2)}{N+\p-2+2s} C_x (r_x^{2-2s}-\epsilon^{2-2s})\cdot \epsilon^{2s}m_\epsilon
\\ & \quad + \frac{(1-s)(N+\p-2)}{N+\p-2+2s}\big(r_x^{-2s}+\frac{2s}{2s-1}r_x^{1-2s}\big)
\cdot \epsilon^{2s}\omega_u(m_\epsilon) \\ & \quad 
+ \frac{s(N+\p-2)}{N+\p-2+2s}\cdot \epsilon^2 \sup_{y\in B_\epsilon}\big|\nabla^2u(y)-\nabla^2u(x)\big|
+\frac{2s(\p-2)}{N+\p-2+2s}\cdot\epsilon^3\frac{|\nabla^2 u(x)|^2}{|p_x|},
\end{split}
\end{equation*}
where the quantity $m_\epsilon$ is as in the statement of Lemma \ref{prop3}.
\end{theorem}
\begin{proof}
We sum up formulas in Lemma \ref{prop3} and Proposition \ref{prop5}, and multiply the
result by the factor:
$\frac{(1-s)s}{C(N,s)|A_\p|}\cdot\frac{N+\p-2}{N+\p-2+2s}\cdot \epsilon^{2s}$. 
\end{proof}

\begin{remark}\label{rem7}
\begin{enumerate}[leftmargin=7mm]
\item[(i)] Analysis similar to Remark \ref{rem4} allows for computing
  the order of the error bound in Theorem \ref{thm6} when $u\in
  C^{0,\alpha}\big(\R^N\setminus \bar B_{r_x}\big)$. In 
  particular, when $\alpha=1$ then the bounding quantity becomes:
\begin{equation*}
\begin{split}
&C_{N,\p,s} \cdot \Big(C\big(C_x, \frac{1}{|p_x|}, \mbox{Lip}_u\big)\cdot C(r_x)
\epsilon^{4s-1} + C\big(|\nabla^2 u(x)|,
\frac{1}{|p_x|}\big)\epsilon^3 +  o(\epsilon^2)\Big).
\end{split}
\end{equation*}
When additionally $u\in C^{2,1}(B_{r_x})$,
the above quantity has order $\epsilon^{4s-1}+\epsilon^3$, which
further reduces to $\epsilon^3$ when $s=1$.

\item[(ii)] For a more precise analysis of the asymptotic expansion when $s\to 1-$, note that:
\begin{equation*}
\begin{split}
& \kappa_\epsilon\leq \sup\Big\{ m;~ m\in [0,2] \; \mbox{ and } \; m^2\leq
 \frac{32\;\omega_u(m)}{|p_x|}\cdot\frac{N+\p-2}{\p-1}\cdot\frac{
  r_x^{-2s} + r_x^{1-2s}}{2s-1}\epsilon^{2s-1}\Big\}, \\ &
\frac{8 \|\nabla^2 u\|_{L^\infty(B_{r_x})}}{|p_x|}\cdot\frac{2s-1}{1-s}\cdot 
\frac{r_x^{2-2s}-\epsilon^{2-2s}}{\epsilon^{1-2s} -{r_x}^{1-2s}}
\leq \frac{8 \|\nabla^2u\|_{L^\infty(B_{r_x})}}{|p_x|}\cdot
16 r_x^{2-2s}|\ln\epsilon|\epsilon^{2s-1}.
\end{split}
\end{equation*}
The first bound above is valid when $\epsilon<\frac{r_x}{2}$ (see
Remark \ref{rem4} (i)), while for the second bound we used that:
$r_x^{2-2s}-\epsilon^{2-2s}\leq (2-2s) (\ln r_x - \ln\epsilon)
r_x^{2-2s}\leq 4(1-s)|\ln\epsilon| r_x^{2-2s}$, valid for all $\epsilon<e^{-|\ln r_x|}$.
It is thus clear that $m_\epsilon \leq o(1)$ as $\epsilon\to 0+$, independently
of $s\in (\frac{1}{2},1)$ bounded away from $\frac{1}{2}$.
In particular, for each fixed $\epsilon$, the bounding quantity in
Theorem \ref{thm6} converges to: 
$$2\frac{\p-2}{N+\p}\cdot \epsilon^3\frac{|\nabla^2 u(x)|^2}{|p_x|}
+ \frac{N+\p-2}{N+\p} \cdot \epsilon^2\sup_{y\in B_\epsilon}\big|\nabla^2
u(y)-\nabla^2 u(x)\big|$$
as $s\to 1-$, which is consistent with (\ref{otto}).
\end{enumerate}
\end{remark}

\section{The averaging operator $\A_\epsilon$ and its dynamic
  programming principle}\label{sec_dppe}

Let $\D\subset\R^N$ be open, bounded and let $F:\R^N\to \R$ be
bounded, Borel. In this section we discuss the non-local Dirichlet-type problem in:
\begin{equation}\label{dppe}\tag*{${\mathrm{(DPP)}}_\epsilon$}
u(x) = \left\{\begin{array}{ll} \A_\epsilon u(x) &
    ~~\mbox{for } x\in\D\\ F(x) &  ~~\mbox{for } x\in\R^N\setminus \D.
\end{array}\right.
\end{equation}

Equivalently, the above equation can be written as $u=S_\epsilon u$, where
the operator $S_\epsilon$ applied on a bounded Borel function
$v:\R^N\to\R$ returns the bounded Borel function:
\begin{equation}\label{Se}
S_\epsilon v = \One_{\D} \cdot \A_\epsilon v + \One_{\R^N\setminus \D} \cdot F 
\end{equation}

\smallskip

The main result of this section is the following observation:

\begin{theorem}\label{th_exists}
For any bounded Borel data $F:\R^N\to \R$, the problem \ref{dppe}
has the unique bounded Borel solution $u_\epsilon^F:\R^N\to\R$ and there
holds: $\|u_\epsilon^F\|_{L^\infty}\leq \|F\|_{L^\infty}$. Moreover, the
solution operator to \ref{dppe} is monotone, that is $F\leq \bar F$
implies $u_\epsilon^F\leq u_\epsilon^{\bar F}$.
\end{theorem}

\smallskip

Before the proof, we derive another useful property:

\begin{lemma}\label{conti}
Let $u:\R^N\to \R$ be bounded Borel. Then the following functions of $x$:
$$\inf_{|y|=1} \fint_{T_\p^{\epsilon, \infty}(y)}  u(x+z)\dmN,\qquad
    \sup_{|y|=1} \fint_{T_\p^{\epsilon, \infty}(y)}  u(x+z)\dmN, \qquad \A_\epsilon u,$$
are uniformly continuous on $\R^N$.
\end{lemma}
\begin{proof}
Denote any of the three listed functions by $f$ and observe that, for
a fixed $x,\bar x\in\R^N$ satisfying $|x-\bar x| <1$ there holds:
\begin{equation}\label{dieci}
|f(x) - f(\bar x)| \leq \sup_{|y|=1}
\Big|\fint_{T_\p^{\epsilon, \infty}(y)}  u(x+z)\dmN -
\fint_{T_\p^{\epsilon, \infty}(y)}  u(\bar x+z)\dmN\Big|
\end{equation}
For any $|y|=1$, we  may write:
\begin{equation*}
\begin{split}
& \Big|\fint_{T_\p^{\epsilon, \infty}(y)}  u(x+z)\dmN - \fint_{T_\p^{\epsilon, \infty}(y)}  u(\bar x+z)\dmN\Big|
\\ &\qquad =\frac{C(N,s)}{\mu_s^N(T_\p^{\epsilon, \infty})}\cdot
\Big|\int_{x+T_\p^{\epsilon, \infty}(y)} \frac{u(z)}{|x-z|^{N+2s}}\;\mathrm{d}z  -
\int_{\bar x+T_\p^{\epsilon, \infty}(y)}  \frac{u(z)}{|\bar
  x-z|^{N+2s}}\;\mathrm{d}z \Big|\\ & \qquad \leq  \frac{C(N,s) \|u\|_{L^\infty}}{\mu_s^N(T_\p^{\epsilon, \infty})}\cdot 
\Big( \int_{(x+T_\p^{\epsilon, \infty}(y))\cap (\bar x+T_\p^{\epsilon,
      \infty}(y)) } \Big|\frac{1}{|x-z|^{N+2s}} - \frac{1}{|\bar x-z|^{N+2s}} \Big|\;\mbox{d}z
\\ & \qquad\quad  + \int_{(x+T_\p^{\epsilon, \infty}(y))\setminus (\bar x+T_\p^{\epsilon,
      \infty}(y)) } \frac{1}{|x-z|^{N+2s}} \;\mbox{d}z
+ \int_{(\bar x+T_\p^{\epsilon, \infty}(y))\setminus (x+T_\p^{\epsilon,
      \infty}(y)) } \frac{1}{|\bar x-z|^{N+2s}} \;\mbox{d}z\Big).
\end{split}
\end{equation*}
We now estimate the three last integral terms above. Since $|\bar
x-z|>\frac{1}{2}|x-z|$ whenever $|x-z|>2$, it follows that for any
$r\geq 2$ there holds:
\begin{equation}\label{undici}
\begin{split}
&\int_{(x+T_\p^{\epsilon, \infty}(y))\cap (\bar x+T_\p^{\epsilon,
      \infty}(y)) } \Big|\frac{1}{|x-z|^{N+2s}} - \frac{1}{|\bar x-z|^{N+2s}} \Big|\;\mbox{d}z \\ & \qquad \qquad
\leq \int_{x+T_\p^{r, \infty} }\frac{1+2^{N+2s}}{|x-z|^{N+2s}} \;\mbox{d}z
+ \int_{x+T_\p^{\epsilon, r} }\frac{N+2s}{\epsilon^{N+1+2s}} |x-\bar x| \;\mbox{d}z
\\ & \qquad \qquad \leq \frac{1+2^{N+2s}}{2s}|A_\p|\cdot r^{-2s} + \frac{(N+2s)
  r^N}{N\epsilon^{N+1+2s}}|A_\p|\cdot |x-\bar x|,
\end{split}
\end{equation}
where we also used that: $\big||x-z|^{-N-2s} - |\bar
x-z|^{-N-2s} \big|\leq \frac{N+2s}{\epsilon^{N+1+2s}} |x-\bar x|$ for
$z\in (x+T_\p^{\epsilon, \infty}(y))\cap (\bar x+T_\p^{\epsilon, \infty}(y))$.
On the other hand, we have:
\begin{equation}\label{dodici}
\begin{split}
&\int_{(x+T_\p^{\epsilon, \infty}(y))\setminus (\bar x+T_\p^{\epsilon,
      \infty}(y)) } \frac{1}{|x-z|^{N+2s}} \;\mbox{d}z
\\ & \qquad \qquad \leq \frac{|A_\p|}{2s}r^{-2s} + \big|T_\p^{\epsilon, r}(y)\setminus
((\bar x - x)+T_\p^{\epsilon, \infty}(y))\big|\cdot\frac{1}{\epsilon^{N+2s}}
\end{split}
\end{equation}
Further, it easily follows that:
$$\sup_{|y|=1} \big|T_\p^{\epsilon, r}(y)\setminus ((\bar x - x)+T_\p^{\epsilon, \infty}(y))\big|
\leq \sup_{|z|<|x-\bar x|} \big|T_\p^{\epsilon, r}\setminus
(z+T_\p^{\epsilon, \infty})\big| \leq \big|(\partial T_\p^{\epsilon,
  r}) + B_{|x-\bar x|}\big|.$$

In conclusion (\ref{dieci}) becomes, in view of (\ref{undici}) and (\ref{dodici}):
\begin{equation*}
\begin{split}
\big|f(x) - f(\bar x)\big|  \leq \frac{C(N,s) \|u\|_{L^\infty}}{\mu_s^N(T_\p^{\epsilon, \infty})}\cdot 
\Big(& \frac{3+2^{N+2s}}{2s}|A_\p|\cdot r^{-2s} + \frac{(N+2s)
  r^N}{N\epsilon^{N+1+2s}}|A_\p|\cdot |x-\bar x| \\ & +  \big|(\partial T_\p^{\epsilon,
  r}) + B_{|x-\bar x|}\big|\cdot\frac{2}{\epsilon^{N+2s}} \Big).
\end{split}
\end{equation*}
It is clear that by taking $r$ large and then $|x-\bar x|$
appropriately small, the right hand side above can be bounded by any
$\delta>0$. This proves the claimed uniform continuity.
\end{proof}

\medskip

\noindent {\bf Proof of Theorem \ref{th_exists}}

Define $v_0\equiv \inf F$ and set $v_n\doteq (S_\epsilon)^n
v_0$, where the operator $S_\epsilon$ is as in (\ref{Se}).
Each function $v_n:\R^N\to \R$ is continuous in $\D$ by Lemma
\ref{conti} and hence Borel in $\R^N$. The sequence  
$\{v_n\}_{n=1}^\infty$ is uniformly bounded: $\|v_n\|_{L^\infty }\leq
\|F\|_{L^\infty}$ and nondecreasing because $v_0\leq v_1$, and
$S_\epsilon$ is order-preserving. Thus, $\{v_n\}_{n=1}^\infty$ has a
pointwise limit $v:\R^N\to \R$ which is bounded Borel and obeys the
same bound: $\|v\|_{L^\infty }\leq \|F\|_{L^\infty}$.

We now show that one can take $u_\epsilon^F=v$. Indeed, for every $x\in\D$ there holds:
\begin{equation*}
\begin{split}
|v_{n+1}(x) - S_\epsilon v(x) | & = |S_\epsilon v_n(x) - S_\epsilon
v(x)|\leq \sup_{|y|=1} \fint_{T_\p^{\epsilon, \infty}(y)}|(v_n-v)(x+z)|\dmN
\\ & \leq \frac{1}{\mu_s^N(T_\p^{\epsilon, \infty})}\int_{\R^N\setminus B_\epsilon}|(v_n-v)(x+z)|\dmN
~~ \to 0 \quad\mbox{ as } n\to\infty,
\end{split}
\end{equation*}
by the monotone convergence theorem. We thus obtain: $v=S_\epsilon v$
on $\D$, as claimed.

To prove uniqueness, assume that $v, \bar v$ are two bounded, Borel
solutions of \ref{dppe}. Clearly $v=\bar v$ on $\R^N\setminus \D$
and denote $M\doteq \sup_\D |v-\bar v|$.  For every $x\in \D$ we
observe that:
$$|(v-\bar v)(x)| = |\A_\epsilon v(x) - \A_\epsilon\bar v(x)|\leq
\sup_{|y|=1}\fint_{x+T_\p^{\epsilon, \infty}(y)} |(v - \bar v)(z)|\dmN
\leq M\cdot \frac{\mu_s^N(T_\p^{\epsilon, \mathrm{diam}\D})}{\mu_s^N(T_\p^{\epsilon, \infty})}.$$
Hence: $M\leq M \cdot \frac{\mu_s^N(T_\p^{\epsilon,
    \mathrm{diam}\D})}{\mu_s^N(T_\p^{\epsilon, \infty})}$, so there
must be $M=0$ and thus $v=\bar v$ in $\D$.

Finally, the claimed monotonicity of the solution operator to
\ref{dppe} follows from the monotonicity of $S_\epsilon$. The proof is done.
\endproof

\section{Convergence to viscosity solutions of $\Delta_\p^s u = 0$.}\label{sec_visc}

In this section we will identify the limits of solutions to the
non-local dynamic programming principle \ref{dppe} in the vanishing removed
singularity radius $\epsilon\to 0$, as viscosity solutions to the
homogeneous Dirichlet problem for $\Delta_\p^s$.

\begin{theorem}\label{th_viscosity}
Let $F:\R^N\to\R$ be uniformly continuous and bounded, and let $\D$ be
an open, bounded subset of $\R^N$. Assume that $\{u_\epsilon\}_{\epsilon\in J}$ is a sequence of solutions 
to \ref{dppe} which converges as $\epsilon\to 0$, $\epsilon \in J$
uniformly,  to some continuous limit function $u\in C(\R^N)$. Then for every
$x\in \D$, $r>0$ and every $\phi\in C^2(\R^N)$ such that $\phi(x) = u(x)$ and
$\nabla \phi(x)\neq 0$, we have:

\begin{enumerate}[leftmargin=7mm]
\item[(i)] if $\phi > u$ on $\bar B_r(x)\setminus \{x\}$, then
  $\mathcal{L}_{s,\p}[\tilde \phi](x)\geq 0$, 
\item[(ii)] if $\phi < u$ on $\bar B_r(x)\setminus \{x\}$, then
  $\mathcal{L}_{s,\p}[\tilde \phi](x)\leq 0$, 
\end{enumerate}
where we denoted: $\tilde\phi = \One_{\bar B_r(x)}\cdot\phi +\One_{\R^N\setminus \bar B_r(x)}\cdot u$.
\end{theorem}

\begin{remark}\label{remi_visc}
Recall by Remark \ref{rem2.5} that $\mathcal{L}_{s,\p}$
differs from $\Delta_\p^s$ only by a multiplicative constant,
depending on $N, s, \p$.  It is also clear that $u$ as in Theorem \ref{th_viscosity} satisfies $u=F$
in $\R^N\setminus \D$. A function $u$ satisfying only
the one-sided comparison with test functions (i) (rather than both
conditions (i) and (ii)) is  called a viscosity subsolution to the non-local Dirichlet problem:
\begin{equation}\label{nD}
\Delta_\p^s u = 0 \quad \mbox{in } \D, \qquad u=F \quad \mbox{in } \R^N\setminus \D.
\end{equation}
When (i) is replaced by (ii), then $u$ is called a viscosity
supersolution. Satisfaction of both conditions is referred to $u$
being a viscosity solution of (\ref{nD}). See also \cite[Definition 4.4]{BCF2}.
\end{remark}

\medskip

\noindent {\bf Proof of Theorem \ref{th_viscosity}}

{\bf 1.} Let $\phi, r$ and $x\in\D$ be as indicated. We will show that
(ii) holds, while the property (i) can be deduced by a symmetric argument. 

For all $j\in\N$ such that $\bar B_{1/j}(x)\subset\D$, define
$\epsilon_j>0$ by requesting that:
$$\|u_\epsilon - u\|_{L^\infty}\leq \frac{1}{2} \min_{\bar
  B_r(x)\setminus B_{1/j}(x)} (u-\phi) \qquad\mbox{for all }\;
\epsilon\leq \epsilon_j,~ \epsilon\in J.$$
Without loss of generality, the sequence $\{\epsilon_j\}_{j\to\infty}$
is decreasing to $0$. Let $\{x_\epsilon\in\D\}_{\epsilon\in J}$ be a
sequence with the property that:
$$(u_\epsilon - \phi)(x_\epsilon) = \min_{\bar B_{1/j}(x)} (u_\epsilon-\phi)
\qquad \mbox{and} \qquad x_\epsilon\in\bar B_{1/j}(x) ~~ \mbox{ for
  all  }\; \epsilon\in (\epsilon_{j+1}, \epsilon_j]\cap J.$$
Then, for all $\bar x\in \bar B_r(x)\setminus B_{1/j}(x)$ we have:
\begin{equation*}
\begin{split}
(u_\epsilon - \phi)(\bar x) & \geq (u-\phi)(\bar x) - \frac{1}{2}\min_{\bar
  B_r(x)\setminus B_{1/j}(x)} (u-\phi)\geq  \frac{1}{2}\min_{\bar
  B_r(x)\setminus B_{1/j}(x)} (u-\phi) \\ & \geq (u_\epsilon - u)(x) =
(u_\epsilon-\phi)(x) \geq (u_\epsilon-\phi)(x_\epsilon).
\end{split}
\end{equation*}
This implies:
\begin{equation}\label{tredici}
(u_\epsilon-\phi)(x_\epsilon) = \min_{\bar B_r(x)} (u_\epsilon-\phi) ~~
\mbox{ for all }\; \epsilon\in J \qquad\mbox{and} \qquad x_\epsilon\to x
~~ \mbox{ as }\; \epsilon\to 0 ,~\epsilon\in J.
\end{equation}

\smallskip

{\bf 2.} Since $u_\epsilon$ satisfies \ref{dppe}, it follows that:
\begin{equation*}
\begin{split}
\A_\epsilon\tilde\phi (x_\epsilon) - \tilde\phi(x_\epsilon) & =
\big(\A_\epsilon\tilde\phi (x_\epsilon) - \tilde\phi(x_\epsilon) \big)
- \big(\A_\epsilon u_\epsilon (x_\epsilon) - u_\epsilon (x_\epsilon)\big) 
\\ & \leq \sup_{|y|=1} \fint_{T_\p^{\epsilon, \infty}(y)}\tilde
\phi(x_\epsilon +z) - u_\epsilon(x_\epsilon+z) - \phi(x_\epsilon) + u_\epsilon(x_\epsilon)\dmN.
\end{split}
\end{equation*}
We fix $|y|=1$ and estimate the integral above. By (\ref{tredici}) it follows that $\tilde
\phi(x_\epsilon+z)-u_\epsilon(x_\epsilon+z) \leq
\phi(x_\epsilon)-u_\epsilon(x_\epsilon)$ whenever $x_\epsilon+z\in\bar
B_r(x)$, so the said integral is bounded by:
\begin{equation*}
\begin{split}
 \frac{1}{\mu_s^N(T_\p^{\epsilon, \infty})}&\int_{T_\p^{\epsilon,
    \infty}(y)\setminus \bar B_r(x-x_\epsilon)} u(x_\epsilon+z) -
u_\epsilon(x_\epsilon+z) - \phi(x_\epsilon) +
u_\epsilon(x_\epsilon)\dmN \\ & \leq \frac{\mu_s^N(T_\p^{\epsilon,
    \infty}(y)\setminus \bar B_r(x-x_\epsilon))}{\mu_s^N(T_\p^{\epsilon, \infty})}
\cdot \big(2\|u-u_\epsilon\|_{L^\infty} + |u(x_\epsilon) - \phi(x_\epsilon)|\big).
\end{split}
\end{equation*}
Hence we get:
\begin{equation}\label{quattordici}
\begin{split}
\A_\epsilon\tilde\phi (x_\epsilon) - \tilde\phi(x_\epsilon) & \leq 
\frac{\mu_s^N(T_\p^{r/2, \infty})}{\mu_s^N(T_\p^{\epsilon, \infty})}
\cdot \big(2\|u-u_\epsilon\|_{L^\infty} + |u(x_\epsilon) - u(x)|+
|\phi(x_\epsilon) -\phi(x)|\big) \\ & = o(\epsilon^{2s})\quad \mbox{ as }\; \epsilon\to 0.
\end{split}
\end{equation}

\smallskip

{\bf 3.} In this step, we show that:
\begin{equation}\label{diciasette}
\big|\mathcal{L}_{s,\p}^\epsilon[\tilde \phi](x_\epsilon) -
\mathcal{L}_{s,\p}[\tilde \phi](x_\epsilon)\big| = o(1)\quad \mbox{ as }\; \epsilon\to 0.
\end{equation}
The argument relies on verifying the proof of Lemma \ref{prop3}. With
the parallel notation $m=\big|\frac{p_x^\epsilon}{|p_x^\epsilon|} - y_\delta^\epsilon\big|$,
where $p_x^\epsilon=\nabla \phi(x_\epsilon)\neq 0$, and where the unit
vector $y_\delta^\epsilon$ is an almost
maximizer of the function $y\mapsto \int_{T_\p^{\epsilon, \infty}(y)}\tilde\phi(x_\epsilon+z)\dmN$, we get the following
  replacement of (\ref{tre.6}):
\begin{equation*}
\begin{split}
\delta & \geq \int_{T_\p^{\epsilon,r-|x_\epsilon
    -x|}(\frac{p_x^\epsilon}{|p_x^\epsilon|})}\phi(x_\epsilon +z) -
\phi\big (x_\epsilon+R_{y_\delta^\epsilon,
  \frac{p_x^\epsilon}{|p_x^\epsilon|}}z\big)\dmN \\ & \quad - 
\omega_u(m)\cdot \int_{T_\p^{r+|x_\epsilon-x|, \infty}}(1+|z|)\dmN -
2\|\tilde \phi\|_{L^\infty}\cdot \mu_s^N\big(T_\p^{r-|x_\epsilon-x|, r+|x_\epsilon-x|}\big)
\\ & \geq \Big\langle p_x^\epsilon, \big(Id_N - R_{y_\delta^\epsilon,
  \frac{p_x^\epsilon}{|p_x^\epsilon|}}\big) \int_{T_\p^{\epsilon,r-|x_\epsilon
    -x|}(\frac{p_x^\epsilon}{|p_x^\epsilon|})} z \dmN \Big\rangle -
O(1) m -  O(1) \omega_u(m) - O(1) |x_\epsilon - x| \\ & \geq
\frac{m^2|p_x^\epsilon|}{2}\int_{T_\p^{\epsilon, r/2}}\langle z,e_1\rangle \dmN - O(1),
\end{split}
\end{equation*}
where $O(1)$ depends on $N,s,\p,r$ and $\|\tilde \phi\|_{L^\infty}$,
$\|\nabla^2\phi\|_{L^\infty(\bar B_r(x))}$. Consequently:
\begin{equation}\label{ds1}
m^2\leq \frac{O(1)}{|p_x^\epsilon|\int_{T_\p^{\epsilon, r/2}}\langle
  z,e_1\rangle \dmN}\leq \frac{O(1)}{|p_x^\epsilon|}\epsilon^{2s-1}.
\end{equation}

Further, as in (\ref{sette}) we get:
$$\Big|\mathcal{L}_{s,\p}^\epsilon[\tilde\phi](x_\epsilon) -
\int_{T_\p^{\epsilon, \infty}(\frac{p_x^\epsilon}{|p_x^\epsilon|})}L_{\tilde\phi}(x_\epsilon,
z,z)\dmN\Big|\leq \delta +J_1 +J_2+J_3,$$
where:
\begin{equation*}
\begin{split}
J_1 & \leq \Big|\int_{T_\p^{\epsilon, r-|x_\epsilon-x|}(\frac{p_x^\epsilon}{|p_x^\epsilon|})}
\phi(x_\epsilon+z) - \phi \big(x_\epsilon + R_{y_\delta^\epsilon, \frac{p_x^\epsilon}{|p_x^\epsilon|}}z\big)
+\phi(x_\epsilon -z) - \phi \big(x_\epsilon - R_{y_\delta^\epsilon, \frac{p_x^\epsilon}{|p_x^\epsilon|}}z\big)
\dmN\Big| \\ & \leq \int_{T_\p^{\epsilon, r-|x_\epsilon-x|}(\frac{p_x^\epsilon}{|p_x^\epsilon|})}
\Big|\big\langle \nabla\phi (x_\epsilon+z) - \nabla\phi(x_\epsilon
-x), \big(Id_N -  R_{y_\delta^\epsilon, \frac{p_x^\epsilon}{|p_x^\epsilon|}}\big)z\big\rangle\Big|\dmN
+ O(1) m^2 \\ & \leq O(1) m,
\end{split}
\end{equation*}
\begin{equation*}
\begin{split}
J_2 & \leq \Big|\int_{T_\p^{\epsilon, r+|x_\epsilon-x|}(\frac{p_x^\epsilon}{|p_x^\epsilon|})}
u (x_\epsilon+z) - u \big(x_\epsilon + R_{y_\delta^\epsilon, \frac{p_x^\epsilon}{|p_x^\epsilon|}}z\big)
+u(x_\epsilon -z) - u \big(x_\epsilon - R_{y_\delta^\epsilon, \frac{p_x^\epsilon}{|p_x^\epsilon|}}z\big)
\dmN\Big| \\ & \leq O(1) \omega_u(m), \\
J_3 & \leq \int_{T_\p^{ r-|x_\epsilon-x|,  r+|x_\epsilon-x|}}
4\|\tilde\phi\|_{L^\infty}\dmN \leq O(1) |x_\epsilon-x|.
\end{split}
\end{equation*}
After passing $\delta\to 0$ and recalling (\ref{ds1}), we conclude:
\begin{equation*}
\Big|\mathcal{L}_{s,\p}^\epsilon[\tilde \phi](x_\epsilon)
-\int_{T_\p^{\epsilon, \infty}(\frac{p_x^\epsilon}{|p_x^\epsilon|})} L_{\tilde\phi}
(x_\epsilon, z,z)\dmN \Big| \leq O(1)\big( m + \omega_u(m) +
|x_\epsilon-x|\big) = o(1) \quad \mbox{ as }\; \epsilon\to 0.
\end{equation*}
The above implies (\ref{diciasette}), in view of: $\big|
\int_{T_\p^{0, \epsilon}(\frac{p_x^\epsilon}{|p_x^\epsilon|})} L_{\tilde\phi}
(x_\epsilon, z,z)\dmN \big| = O(1)\epsilon^{2-2s}$.

\smallskip

{\bf 4.} Recall that by Theorem \ref{thm4.7} we also directly have:
$$\Big|\mathcal{L}_{s,\p}^\epsilon[\tilde \phi](x)
- \mathcal{L}_{s,\p}[\tilde\phi](x)\Big|\leq o(1) \quad \mbox{ as }\; \epsilon\to 0,$$
because $\tilde\phi = \phi$ on $\bar B_r(x)$ has regularity $C^2$,
while $\tilde\phi = u$ on $\R^N\setminus \bar B_r(x)$ is uniformly
continuous. Together with (\ref{diciasette}), this yields:
\begin{equation}\label{quindici}
\begin{split}
\Big|\big(\A_\epsilon\tilde\phi &(x_\epsilon) - \tilde\phi(x_\epsilon)\big) 
- \big(\A_\epsilon \tilde\phi(x) - \tilde\phi (x)\big)\Big|   =
\frac{1}{\mu_s^N(T_\p^{\epsilon, \infty})} \Big|\mathcal{L}_{s,\p}^\epsilon[\tilde \phi](x_\epsilon)
- \mathcal{L}_{s,\p}^\epsilon[\tilde\phi](x)\Big| \\ &=
\frac{1}{\mu_s^N(T_\p^{\epsilon, \infty})} \Big(\Big|\mathcal{L}_{s,\p}[\tilde \phi](x_\epsilon)
- \mathcal{L}_{s,\p}[\tilde\phi](x)\Big| + o(1)\Big) \quad \mbox{ as }\; \epsilon\to 0.
\end{split}
\end{equation}
We now denote $R\doteq R_{\frac{p_x^\epsilon}{|p_x^\epsilon|},\frac{p_x}{|p_x|} }$and  estimate:
\begin{equation*}
\begin{split}
& \Big|\mathcal{L}_{s,\p}[\tilde \phi](x_\epsilon) -
\mathcal{L}_{s,\p}[\tilde\phi](x)\Big| \\ & \leq
\int_{T_\p^{0,\infty}(\frac{p_x}{|p_x|})}\Big|\big(\tilde\phi(x_\epsilon + R z)
+ \tilde\phi (x_\epsilon - R z) - 2\tilde\phi(x_\epsilon)\big) - \big(\tilde \phi(x+z)
+\tilde\phi(x-z)-2\phi(x)\big)\Big|\dmN \\ & \leq \frac{1}{2}\int_{T_\p^{0,r}}|z|^2\dmN 
\cdot \sup_{z\in B_r(x)}\Big|\big( \nabla\phi(x_\epsilon + Rz) + \nabla\phi(x_\epsilon-Rz)\big)
 - \big( \nabla\phi(x + Rz) + \nabla\phi(x-Rz)\big)\Big|\\
&\quad + \int_{T_\p^{r+|x_\epsilon - x|,\infty}}
2\omega_u(|x_\epsilon - x|) + 2\omega_u(|Id_N-R|)\cdot (1+|z|)\dmN 
\\ &\quad + \int_{T_\p^{r-|x_\epsilon - x|, r+|x_\epsilon - x|}} 6\|\tilde\phi\|_{L^\infty}\dmN
\\ &\leq O(1) \cdot\Big(|x_\epsilon - x| + \omega_u(|x_\epsilon-x|) +
\omega_u\big(\big|\frac{p_x^\epsilon}{|p_x^\epsilon|} -
\frac{p_x}{|p_x|}\big|\big)\Big) \leq o(1) \quad \mbox{ as }\; \epsilon\to 0.
\end{split}
\end{equation*}

At this point, combining (\ref{quindici}) with (\ref{quattordici}) results in:
$$\mathcal{A}_\epsilon\tilde \phi(x) - \tilde\phi(x) \leq o(\epsilon^{2s}) \quad \mbox{ as }\; \epsilon\to 0.$$
On the other hand, Theorem \ref{thm4.7} implies that:
$$\mathcal{A}_\epsilon\tilde \phi(x) - \tilde\phi(x)  =
\frac{s}{C(N,s) |A_\p|} \epsilon^{2s}\cdot
\mathcal{L}_{s,\p}[\tilde\phi](x) + o(\epsilon^{2s}).$$
The above two asymptotic statements directly yield
$\mathcal{L}_{s,\p}[\tilde\phi](x)\leq 0$, as claimed.
\endproof

\section{The non-local Tug-of-War game with noise}\label{sec_game}

In this section, we develop the basic probability setting related to the equation \ref{dppe}.

\medskip

{\bf 1.} Consider the probability space $(T_\p^{1,\infty}, \mathcal{B},
\frac{1}{\mu_s^N(T_\p^{1,\infty})}\mu_s^N)$ equipped with the standard Borel
$\sigma$-algebra and the normalised $\mu_s^N$ measure, and define
$(\Omega_1, \mathcal{F}_1, \mathbb{P}_1)$ as the product space with 
the counting measure on the discrete set $\{1,2\}$. In particular,
for every $B\in\mathcal{B}$, we have:
$$\mathbb{P}_1\big(B\times \{1,2\}\big) = \frac{2s}{|A_\p|}\int_{B}\frac{1}{|z|^{N+2s}}\;\mbox{d}z.$$
Further, the countable product of $(\Omega_1, \mathcal{F}_1,
\mathbb{P}_1)$ is denoted by $(\Omega, \mathcal{F}, \mathbb{P})$, where:
\begin{equation*}
\begin{split}
\Omega= (\Omega_1)^{\mathbb{N}} = \big\{\omega=\{(z_i,
s_i)\}_{i=1}^\infty; ~ z_i\in T_\p^{1,\infty},~ s_i\in \{1,2\} ~~\mbox{ for all } i\in\mathbb{N}\big\}.
\end{split}
\end{equation*}
For each $n\in\mathbb{N}$, the probability space $(\Omega_n, \mathcal{F}_n, \mathbb{P}_n)$ is the 
product of $n$ copies of $(\Omega_1, \mathcal{F}_1,
\mathbb{P}_1)$ and the $\sigma$-algebra $\mathcal{F}_n$ is 
identified with the sub-$\sigma$-algebra of $\mathcal{F}$, 
consisting of sets $A\times \prod_{i=n+1}^\infty\Omega_1$
for all $A\in\mathcal{F}_n$. Then $\{\mathcal{F}_n\}_{n=0}^\infty$ where 
$\mathcal{F}_0= \{\emptyset, \Omega\}$, is a filtration of $\mathcal{F}$. 

\medskip

{\bf 2.} Given are two families of functions $\sigma_I=\{\sigma_I^n\}_{n=0}^\infty$ and
$\sigma_{II}=\{\sigma_{II}^n\}_{n=0}^\infty$, defined on the
corresponding spaces of ``finite histories'' $H_n=\R^N\times (\R^N\times\Omega_1)^n$:
$$\sigma_I^n, \sigma_{II}^n:H_n\to \{y\in\R^N;~ |y|=1\},$$
assumed to be measurable with respect to the (target) Borel
$\sigma$-algebra and the (domain) product $\sigma$-algebra on $H_n$.
For every $x\in\R^N$ and $\epsilon\in (0,1)$ we recursively define:
$$\big\{X_n^{\epsilon, x, \sigma_I, \sigma_{II}}:\Omega\to\R^N\big\}_{n=0}^\infty.$$
For simplicity of notation, we often suppress some of the superscripts $\epsilon, x, \sigma_I, \sigma_{II}$
and write $X_n$ 
instead of $ X_n^{\epsilon, x, \sigma_I, \sigma_{II}}$, if no ambiguity arises.
Recall that $R_{\tilde y, y}\in SO(N)$ is as in Definition \ref{def_cone}. We put: 
\begin{equation}\label{processMp}
\begin{split}
& \, X_0\equiv x, \\ & \, X_n\big((z_1,s_1), \ldots,
(z_n,s_n)\big) \doteq X_{n-1} + \left\{\begin{array}{ll} 
\epsilon R_{\sigma_I^{n-1}, e_1}z_n & \mbox{for } s_n=1 \vspace{1mm} \\ 
\epsilon R_{\sigma_{II}^{n-1}, e_1}z_n & \mbox{for } s_n=2. \end{array} \right.\\
\end{split}
\end{equation}
In this ``game'', each of the two players chooses (deterministically) a
direction $y$, according to their ``strategy'' $\sigma_I$ and
$\sigma_{II}$. These choices are activated  by the value of the equally probable outcomes:
$s_n=1$ activates $\sigma_I$ and $s_n=2$ activates $\sigma_{II}$. 
The position $X_{n-1}$ is then advanced by a shift 
$\epsilon R z\in T_\p^{\epsilon, \infty}(y)$, randomly in $z$ according to the
normalised measure $\mu_s^N$ on $T_\p^{1,\infty}$. 

\medskip

{\bf 3.}  Given an open, bounded domain $\D\subset\R^N$, 
define further the $\mathcal{F}$-measurable random variable:
$\tau^{\epsilon, x, \sigma_I, \sigma_{II}}:\Omega\to \mathbb{N}\cup\{+\infty\}$ by:
$$\tau(\omega) \doteq \min\big\{n\geq 1;~ X_n\not\in \D\big\}.$$
We observe that $\tau$ is finite $\mathbb{P}$-a.e., making it a
stopping time. Indeed, since $\PP_1(T_\p^{\mathrm{diam}\D,
  \infty}\times \{1,2\})>0$ it follows that $\PP\big(\omega;~ \exists
i ~~ z_i\in T_\p^{\mathrm{diam}\D, \infty}\big)=1$, and on this event $\tau<\infty$.

Let $F:\R^N\to\mathbb{R}$ be a given bounded, uniformly continuous function.  In our
``game'', the first ``player'' collects from his 
opponent the payoff, given by the data $F$ at the stopping position. The incentive
of the collecting ``player'' to maximize the outcome and of the
disbursing ``player'' to minimize it, leads to the definition of the two game values in:
\begin{equation}\label{15}
\begin{split}
& u_I^\epsilon(x) =
\sup_{\sigma_I}\inf_{\sigma_{II}}\mathbb{E}\Big[F\circ \big(X^{\epsilon, x,
\sigma_I, \sigma_{II}}\big)_{\tau^{\epsilon, x, \sigma_I,  \sigma_{II}}}\Big], \\
& u_{II}^\epsilon(x) =
\inf_{\sigma_{II}}\sup_{\sigma_{I}}\mathbb{E}\Big[F\circ \big(X^{\epsilon, x,
\sigma_I, \sigma_{II}}\big)_{\tau^{\epsilon, x, \sigma_I, \sigma_{II}}}\Big]. 
\end{split}
\end{equation}
It is  clear that $u_{I}^\epsilon$, $u_{II}^\epsilon$ depend only on the values of
$F$ on $\R^N\setminus\D$. We now show that both game values coincide with the unique solution to the
dynamic programming principle \ref{dppe} modeled on the non-local
asymptotic expansion in Theorem \ref{thm4.7}.

\medskip

\begin{lemma}\label{lem_ue}
For each $\epsilon\ll 1$ we have $u_I^\epsilon =
u_{II}^\epsilon=u_\epsilon$, where $u_\epsilon$ is the unique bounded, Borel solution to \ref{dppe}.
\end{lemma}
\begin{proof}
{\bf 1.} We will show that $u_{II}^\epsilon\leq u_\epsilon$, while the
inequality $u_\epsilon\leq u_{I}^\epsilon$ can be proved by a
symmetric argument and  $u_I^\epsilon \leq u_{II}^\epsilon$ is always valid.
Fix $x\in\R^N$ and $\epsilon, \delta>0$. We choose a strategy
$\sigma_{II,0}=\{\sigma_{II,0}^n(X_n)\}_{n=0}^\infty$ satisfying:
\begin{equation}\label{markov}
\inf_{|y|=1} \fint_{T_\p^{\epsilon, \infty}(y)} u_\epsilon(x+z)\dmN + \frac{\delta}{2^{n+1}}
\geq \fint_{T_\p^{\epsilon, \infty}(\sigma_{II,0}^n)} u_\epsilon(x+z)\dmN.
\end{equation}
The fact that such Borel-regular strategy exists follows from Lemma \ref{conti}.
Indeed, let $\{B(x_i, \xi)\}_{i=1}^\infty$ be a locally finite
covering of $\R^N$, where:
$$\Big|\inf_{|y|=1}\fint_{T_\p^{\epsilon, \infty}(y)} u_\epsilon(x+z)\dmN
- \inf_{|y|=1}\fint_{T_\p^{\epsilon, \infty}(y)}  u_\epsilon(\bar
x+z)\dmN\Big|\leq \frac{\delta}{2^{n+2}} \qquad \mbox{for all }\; |x-\bar x|<\xi. $$
For each $i\in\N$ there exists then $|y_i|=1$ with the property:
$\big|\inf_{|y|=1}\fint_{T_\p^{\epsilon, \infty}(y)} u_\epsilon(x_i+z)\dmN - \fint_{T_\p^{\epsilon, \infty}(y_i)}
u_\epsilon(x_i+z)\dmN\big|<\frac{\delta}{2^{n+2}}$. We hence define:
$$\sigma_{II,0}^n(x) = y_i\qquad\mbox{for all }\; x\in B(x_i,
\xi)\setminus \bigcup_{j=1}^{i-1}B(x_j,\xi).$$

\smallskip

{\bf 2.} Fix $x\in\D$ and a strategy $\sigma_I$. Consider the sequence of random variables:
$$M_n\doteq u_\epsilon\circ X_{n\wedge \tau}^{\epsilon, x, \sigma_I,\sigma_{II,0}} + \frac{\delta}{2^n}.$$
We now check that $\{M_n\}_{n=1}^\infty$ is a supermartingale with
respect to the filtration $\{\mathcal{F}_n\}_{n=0}^\infty$. On the
event $n>\tau$ there clearly holds $\mathbb{E}(M_n|\mathcal{F}_{n-1})
=M_{n-1}$. On the other hand, on the event $n\leq\tau$ we have
$X_{n-1}\in\D$, so the property \ref{dppe} and (\ref{markov}) imply that:
\begin{equation*}
\begin{split}
& \mathbb{E}(M_n\mid \mathcal{F}_{n-1}) - M_{n-1} \\ &  =\frac{1}{2}\Big(
\fint_{T_\p^{\epsilon, \infty}(\sigma_{I}^n)} u_\epsilon(X_{n-1}+z)\dmN +
\fint_{T_\p^{\epsilon, \infty}(\sigma_{II,0}^n)} u_\epsilon(X_{n-1}+z)\dmN \Big)
-\frac{\delta}{2^{n+1}} -u_\epsilon(X_{n-1}) \\ & \leq
\mathcal{A}_\epsilon u_\epsilon(X_{n-1}) -u_\epsilon(X_{n-1}) =0.
\end{split}
\end{equation*}
Using Doob's optional stopping theorem, we arrive at:
$$u_\epsilon(x)+\delta =\mathbb{E}[M_0]\geq \mathbb{E}[M_\tau] = \frac{\delta}{2^{\tau}} 
+\mathbb{E}\big[F\circ X_\tau^{\epsilon, x, \sigma_I,
  \sigma_{II,0}}\big]\geq \frac{\delta}{2^{\tau}} 
+ \inf_{\sigma_{II}}\mathbb{E}\big[F\circ X_\tau^{\epsilon, x, \sigma_I, \sigma_{II}}\big].$$
This yields: $u_\epsilon(x)\geq\delta\geq u^\epsilon_{II}(x)$,
concluding the proof in view of $\delta$ being arbitrary.
\end{proof}

\section{Auxiliary estimates for the barrier function}\label{sec_ft}

The purpose of this section is to show the first boundary regularity
estimate for the game process $\{X_n\}_{n=0}^\infty$, towards establishing
our main asymptotic equicontinuity result. 

\begin{theorem}\label{cone_thengood}
Let $\D\subset\R^N$ be an open, bounded domain satisfying the external cone condition.
Namely, assume that there exists a finite cone $C$ such that for
each $x\in\partial \D$ there holds: $x+ S_xC\subset\R^N\setminus \D$, for
some rotation $S_x\in SO(N)$. Then:
\begin{equation}\label{gr1}
\begin{split}
\forall \delta>0\quad\exists \hat\delta<\delta,~
\hat\epsilon>0\quad & \forall \epsilon<\hat\epsilon,~ x\in\partial\D,~ x_0\in B_{\hat\delta}(x)\cap\D \quad 
\\ & \exists \sigma_{I,0}\quad\forall\sigma_{II}\quad
\PP\big(\exists n<\tau \quad  X_n^{\epsilon, x_0,\sigma_{I,0}, \sigma_{II}}\not\in B_\delta(x)\big)\leq\bar\theta,
\end{split}
\end{equation}
with a constant $\bar\theta<1$ depending only on $N,\p,s$ and the cone $C$.
\end{theorem}

The proof will rely on a suitable barrier functions, introduced in \cite{BCF2}. Namely, consider the 
uniformly continuous and bounded $f_t:\R^N\to\R$, where for each $t>0$ we define:
$$f_t(x) = \min\big\{2^t,  |x|^{-t}\big\}.$$
We start by observing a refinement of \cite[Lemma 3.10]{BCF2}:

\begin{proposition}\label{Lspft}
There exists $t_0\gg 1$ depending on $N,\p, s$, such that for all
$t\geq t_0$ and all $|x|\geq 1$ there holds:
$$\mathcal{L}_{s,\p}[f_t](x)\geq C |x|^{-2s-t},$$
with a constant $C$ depending on $N,\p, s$ but not on $t$ or $x$.
\end{proposition}
\begin{proof}
Observe that $\mathcal{L}_{s,\p}[f_t](x)=\mathcal{L}_{s,\p}[f_t](|x| e_1)$  by
rotational invariance. It hence suffices to estimate, after changing variables:
\begin{equation}\label{muno}
\begin{split}
\int_{T_\p^{0,\infty}}\frac{L_{f_t}(|x| e_1, z, z)}{|z|^{N+2s}}\;\mbox{d}z & =
|x|^{-2s}\int_{T_\p^{0,\infty}}\frac{f_t(|x| (e_1+z)) + f_t(|x|( 
  e_1-z))-2f_t(|x| e_1)}{|z|^{N+2s}}\;\mbox{d}z \\ & \geq 
|x|^{-2s-t}\int_{T_\p^{0,\infty}}\frac{f_t(e_1+z) + f_t(e_1-z)-2f_t(e_1)}{|z|^{N+2s}}\;\mbox{d}z, 
\end{split}
\end{equation}
where in the last step above we used that $f_t(|x| z)\geq
|x|^{-t}f_t(z)$ which can be easily checked directly. Further, $L_{f_t}(e_1, z,
z)\geq -2$ for all $z$, while for $|z|\leq \frac{1}{2}$ we have:
\begin{equation*}
\begin{split}
L_{f_t}(e_1, z, z)  & =  |1+|z|^2 +
2\langle e_1, z\rangle|^{-t/2} +  |1+|z|^2 - 2\langle e_1, z\rangle|^{-t/2} -2
\\ & \geq 2(1+|z|^2)^{-t/2} + \frac{t}{2}\big(\frac{t}{2}+1\big)(2\langle
e_1, z\rangle)^2 \big(1+|z|^2\big)^{-t/2-2} - 2 \\ & \geq 2\big(1-\frac{t}{2}|z|^2\big)
+ \frac{t}{2} (t+2) \langle e_1, z\rangle^2 \big(1 - \big(\frac{t}{2}+2\big)|z|^2\big)^{-t/2-2} -2\\ & 
= \frac{t}{2}(t+2) \langle e_1, z\rangle^2 - t|z|^2 - \frac{t}{4}(t+2)(t+4) \langle e_1, z\rangle^2|z|^2,
\end{split}
\end{equation*}
by Taylor's expansion and since $(1+|z|^2)^{-\alpha}\geq 1- \alpha|z|^2$ whenever $\alpha>0$.
Thus (\ref{muno}) becomes:
\begin{equation*}
\begin{split}
& \int_{T_\p^{0,\infty}}  \frac{L_{f_t}(|x| e_1, z, z)}{|z|^{N+2s}}\;\mbox{d}z  \geq
|x|^{-2s-t}\Big(\int_{T_\p^{0,r}} \frac{L_{f_t}(e_1, z, z)
}{|z|^{N+2s}}\;\mbox{d}z - \int_{T_\p^{r, \infty}} \frac{2}{|z|^{N+2s}}\;\mbox{d}z \Big)\\ &
\quad = |x|^{-2s-t} |A_\p|\cdot\Big(
\frac{t}{2}(t+2)\frac{\p-1}{N+\p-2}\cdot\frac{r^{2-2s}}{2-2s} - t \frac{r^{2-2s}}{2-2s} \\ & 
\quad \qquad\quad\qquad\qquad \quad - 
\frac{t}{4}(t+2)(t+4) \frac{\p-1}{N+\p-2}\cdot\frac{r^{4-2s}}{4-2s} - \frac{r^{-2s}}{s} \Big)
\doteq  |x|^{-2s-t} |A_\p|\cdot I_{N,\p, s,t,r}
\end{split}
\end{equation*}
by recalling Lemma \ref{lem_Tp} (i) and where we fixed some
appropriate $r<\frac{1}{2}$. We now estimate the quantity $I_{N,\p, s,t,r}$. 
When $\frac{t+2}{2}\cdot\frac{\p-1}{N+\p-2}\geq 2$, then the first two
terms in $I_{N,\p, s,t,r}$ are bounded from below by:
$\frac{t}{4}(t+2)\frac{\p-1}{N+\p-2}\cdot
\frac{r^{2-2s}}{2-2s}$. Further, when
$(t+4)\frac{r^2}{4-2s}\leq\frac{1}{2(2-2s)}$, then the first three
terms in $I_{N,\p, s,t,r}$ are bounded from below by:
$\frac{t}{8}(t+2)\frac{\p-1}{N+\p-2}\cdot
\frac{r^{2-2s}}{2-2s}$. Finally, when
$\frac{t}{8}(t+2)\frac{\p-1}{N+\p-2}\cdot\frac{r^2}{2-2s}\geq
\frac{2}{s}$, then we have:
$$I_{N,\p, s,t,r}\geq \frac{r^{-2s}}{s}\geq \frac{1}{s}\Big(\frac{2-2s}{2-s} (t+4)\Big)^s.$$
It is clear that the above listed conditions, namely:
$$\frac{t+2}{2}\cdot\frac{\p-1}{N+\p-2}\geq 2\quad\mbox{ and }\quad 
\exists r< \frac{1}{2}:\quad \frac{16
  (2-2s)(N+\p-2)}{s(\p-1)}\cdot\frac{1}{t(t+2)}\leq r^2\leq \frac{2-s}{2-2s}\cdot\frac{1}{t+4}$$
are compatible for sufficiently large $t\geq t_0(N,\p,s)$. The proof is done.
\end{proof}

\begin{corollary}\label{ft_ring}
Let $t_0\gg 1$ be as Proposition \ref{Lspft}. For every $t\geq t_0$
and $R>1$ there exists $\epsilon_0>0$, depending on $N,\p,s,t,R$ such that:
$$\mathcal{A}_{\epsilon}f_t(x)\geq f_t(x) + C \epsilon^{2s} R^{-2s-t}
\qquad\mbox{for all }\; |x|\in [1, R], ~~\epsilon<\epsilon_0,$$
with a constant $C$ depending only on $N,\p, s$.
\end{corollary}
\begin{proof}
We apply Theorem \ref{thm4.7} with $r_x\in \big(\frac{1}{4},
\frac{1}{2}\big)$ and note that $C_x\leq C(N,t)$ and $|p_x|\geq t
R^{-t-1}$ whenever $|x|\in [1,R]$. In view of Proposition \ref{Lspft} it follows that:
$$\mathcal{A}_\epsilon f_t(x) \geq f_t(x) + C_{N,\p,s}\epsilon^{2s}
R^{-2s-t} - C_{N,s,t}\epsilon^2 - C_{N,s,t}\epsilon^{2s}\big(m_\epsilon +\omega_{f_t}(m_\epsilon)\big).$$
Recalling Remark \ref{rem4} (i) we see that: $m_\epsilon\leq
C_{N,\p,s,t}\epsilon^{s-\frac{1}{2}}R^{t+1}$ and $\omega_{f_t}(m_\epsilon)\leq C_tm_\epsilon$.
Hence:
$$\mathcal{A}_\epsilon f_t(x) \geq f_t(x) + C_{N,\p,s}\epsilon^{2s}
\big(R^{-2s-t} - C_{N,s,t}\epsilon^{2-2s} - C_{N,\p,s,t} \epsilon^{s-\frac{1}{2}}R^{t+1} \big),$$
and so the result follows for $C_{N,s,t}\epsilon^{2-2s}  + C_{N,\p,s,t} \epsilon^{s-\frac{1}{2}}R^{t+1} \leq \frac{1}{2}
R^{-2s-t}$.
\end{proof}

\medskip

Towards the proof of Theorem \ref{cone_thengood}, we note that:

\begin{proposition}\label{fe_prob}
Given $R>1$ and $t\geq t_0$, $\epsilon<\epsilon_0$  as in Proposition \ref{Lspft} and Corollary \ref{ft_ring}, let
$v_\epsilon:\R^N\to\R$ be the unique bounded, Borel solution to the problem:
$$ v_\epsilon (x) = \left\{\begin{array}{ll}\mathcal{A}_\epsilon
    v_\epsilon (x) & \mbox{ for } \; |x|\in (1, R)\vspace{1mm}\\ f_t(x) & \mbox{ for } \; |x|\not\in (1, R).
\end{array}\right.$$
Then we have:
\begin{enumerate}[leftmargin=7mm]
\item[(i)] $v_\epsilon \geq f_t$ in $\R^N$,
\item[(ii)] For every $\tilde R\in (1,R)$ exists $\theta_{\tilde R, R}<1$
depending only on $R,\tilde R, N, \p, s$ such that for all $|x|\in [1,\tilde R]$ and  $\epsilon<\epsilon_0$ there holds:
$$\exists \sigma_{I,0}\quad\forall\sigma_{II}\qquad
\mathbb{P}\big(|X^{\epsilon, x, \sigma_{I,0},  \sigma_{II}}_\tau|\geq R\big)\leq\theta_{\tilde R, R},$$
where $\tau$ denotes the first exit time from the annulus $B_{R}(0)\setminus \bar B_1(0)$. 
\item[(iii)] For a given $r>0$, let $\bar \tau = \min\{i\geq 0;~
  |X_i|\not\in (r, rR^2)\}$. Then, for every $|x|\in [r,rR]$ and
$\epsilon<r\epsilon_0$  there holds, with a constant $\theta_R<1$
depending only on $R, N,\p,s$:
$$\exists \sigma_{I,0}\quad\forall\sigma_{II}\qquad
\PP \big(|X^{\epsilon, x, \sigma_{I,0}, \sigma_{II}}_{\bar \tau}|\geq rR^2\big)\leq\theta_R.$$
\end{enumerate}
\end{proposition}
\begin{proof}
{\bf 1.} To show (i), observe that by Corollary \ref{ft_ring} we have, for all $|x|\in (1,R)$:
\begin{equation}\label{bda}
\begin{split}
v_\epsilon(x) - f_t(x) & = \big(\mathcal{A}_\epsilon v_\epsilon (x) -
\mathcal{A}_\epsilon f_t(x)\big) + \big( \mathcal{A}_\epsilon f_t(x) -
f_t(x)\big) \\ & \geq \inf_{|y|=1} \fint_{T_\p^{\epsilon,\infty}(y)}
(v_\epsilon -f_t)(x+z)\dmN + C_{N,\p,s}\epsilon^{2s} R^{-2s-t}\\ &
\geq \inf_{\R^N} (v_\epsilon-f_t) + C_{N,\p,s}\epsilon^{2s} R^{-2s-t}.
\end{split}
\end{equation}
Assume that $M_\epsilon \doteq \inf_{\D} (v_\epsilon-f_t)<0$, in which
case there also holds: $M_\epsilon =\inf_{\R^N} (v_\epsilon-f_t)$. Let
$\{x_n\}_{n=1}^\infty$ me a minimizing sequence in $\D$. Applying
(\ref{bda}) at each $x_n$ and passing to the limit $n\to\infty$, it
follows that: $M_\epsilon>M_\epsilon$, which is a contradiction.

\smallskip

{\bf 2.} To show (ii), fix $t=t_0$ and recall that (i) implies:
$$0\leq v_\epsilon(x) -f_t(x)
=\sup_{\sigma_I}\inf_{\sigma_{II}}\mathbb{E}\big[f_t\circ X^{\epsilon, x, \sigma_{I},  \sigma_{II}}_\tau
- f_t(x)\big].$$
Since $\frac{\tilde R^{-t} - R^{-t}}{2}>0$, it follows that
there exists $\sigma_{I,0}$ such that for all $\sigma_{II}$ there holds:
\begin{equation*}
\begin{split}
- \frac{\tilde R^{-t} - R^{-t}}{2}&\leq \mathbb{E}\big[f_t\circ X^{\epsilon, x, \sigma_{I,0},  \sigma_{II}}_\tau
- f_t(x)\big] \\ & = \int_{\{|X_\tau|\geq R\}} f_t(X_\tau) -
f_t(x)\;\mbox{d}\PP + \int_{\{|X_\tau|<\leq 1\}} f_t(X_\tau) - f_t(x)\;\mbox{d}\PP 
\\ & \leq \PP(|X_\tau|\geq R) \big(R^{-t} - \tilde R^{-t} \big) +
(1-\PP(|X_\tau|\geq R)) \big(2^{t} - \tilde R^{-t} \big)
\\ & = \PP(|X_\tau|\geq R) \big(R^{-t} - 2^t\big) + 2^t - \tilde R^{-t}.
\end{split}
\end{equation*}
Consequently, we obtain: $\PP(|X_\tau|\geq R)\leq \frac{\frac{1}{2} (\tilde R^{-t} - R^{-t}) 
+ 2^t -\tilde R^{-t}}{2^t- R^{-t}} = \frac{ 2^t - \frac{1}{2} (\tilde
R^{-t} + R^{-t})}{2^t- R^{-t}} \doteq\theta_{\tilde R, R}<1.$

The statement in (iii)  follows by scaling invariance after applying
(ii) to $R<R^2$ in place of $\tilde R<R$, so that $\theta_R\doteq\theta_{\tilde R, R}$. This ends the proof.
\end{proof}

\medskip

We finally are ready to give:

\smallskip

\noindent {\bf Proof of Theorem \ref{cone_thengood}}

The cone condition implies existence of  $d>1$ and $\bar r>0$ such
that for all $r<\bar r$ there is a ball $B_r(\bar x)\subset \R^N\setminus
\D$, centered at $\bar x$ with $|x-\bar x|=rd$. Define $R=2d-1$, so that $x\in B_{rR}(\bar x)$.

Given $\delta>0$, let $r<\bar r$ be such that: $\delta\geq rR^2 + rd =
r\big(R^2 + \frac{R+1}{2}\big)$. Letting $\hat\delta =rd$ we get:
\begin{equation}\label{g}
B_{\hat\delta}(x)\subset B_{rR}(\bar x)\setminus \bar
B_r(x)\quad\mbox{ and }\quad B_{rR^2}(\bar x) \subset B_\delta(x).
\end{equation}
Fix $x_0\in B_{\hat\delta}(x)\cap\D$ and $\epsilon< r\epsilon_0$,
where $\epsilon_0$ is as in Proposition \ref{fe_prob} (iii). Denote
by $\bar\tau$ the exit time from the annulus $B_{rR^2}(\bar x) \setminus \bar B_r(x)$.
Then, there exists $\sigma_{I,0}$ such that for all $\sigma_{II}$ there holds:
$$\PP\big(\exists n<\tau \quad  X_n^{\epsilon, x_0,\sigma_{I,0},
  \sigma_{II}}\not\in B_\delta(x)\big)\leq 
\PP\big(X_{\bar\tau} \not\in B_{rR^2}(\bar x)\big)\leq \theta_R\doteq\bar\theta. $$
The first inequality above follows from (\ref{g}), while the second
inequality is a direct consequence of Proposition \ref{fe_prob} (iii).
\endproof

\section{Approximate equicontinuity of solutions to \ref{dppe}}\label{sec_convergence}

In this section, we assume that the open, bounded domain
$\D\subset\R^N$ satisfies the external cone condition. Our goal is
to show that the family $\{u_\epsilon\}_{\epsilon\to 0}$ of solutions
to \ref{dppe} with a given bounded, uniformly continuous $F:\R^N\to \R$
is then approximately equicontinuous, namely:
\begin{equation}\label{ecpe}
\forall\xi>0\quad \exists \hat\epsilon, \delta>0 \quad 
\forall \epsilon\in (0,\hat\epsilon)\quad \forall x, \bar
x\in\R^N\qquad  |x-\bar
x|<\delta \implies |u_\epsilon(x) - u_\epsilon(\bar x)|\leq\xi.
\end{equation}
Together with the uniform boundedness of the family $\{u_\epsilon\}_{\epsilon\to
  0}$, the above condition yields, via the Ascoli-Arzel\`a theorem, that
every sequence in the said family has a further subsequence, converging
uniformly as $\epsilon\to 0$ to some continuous limit function.

\begin{lemma}\label{bdary_enough}
Condition (\ref{ecpe}) is implied by the following weaker
equicontinuity statement:
\begin{equation}\label{ecpe1}
\begin{split}
\forall\xi>0\quad \exists \hat\epsilon, \delta>0 \quad 
\forall \epsilon\in (0,\hat\epsilon)\quad & \forall x\in\D,~ \bar
x\in\partial\D\qquad \\ &  |x-\bar x|<\delta \implies |u_\epsilon(x) - u_\epsilon(\bar x)|\leq\xi.
\end{split}
\end{equation}
\end{lemma}
\begin{proof}
Since $u_\epsilon = F$ on $\R^N\setminus\D$, we get (\ref{ecpe}) for
$x,\bar x\not\in\D$ in view of the uniform continuity of $F$. By
(\ref{ecpe1}), it suffices to consider the case $x,\bar x\in\D$. Fix
$\xi>0$ and choose $\hat\epsilon, \hat\delta>0$ such that:
$$\epsilon\in (0,\hat\epsilon),~~ z\in\R^N\setminus \D^{\hat\delta},~~
|w|\leq\hat\delta \implies |u_\epsilon(z) - u_\epsilon(z+w)|\leq \xi,$$
where we denoted the inner set:
$$\D^{\hat\delta} = \big\{x\in \D;~ \mathrm{dist}(x,\R^N\setminus\D)>\hat\delta\big\}.$$

Fix $x,\bar x\in \D$ such that $|x-\bar x|<\frac{\hat\delta}{2}$
and consider the function $\bar u_\epsilon:\R^N\to\R$ given by:
$$\bar u_\epsilon(z) \doteq u_\epsilon\big(z-(x-\bar x)\big)+\xi. $$
Observe that $\bar u_\epsilon$ solves \ref{dppe} on $\D^{\hat\delta}$, and
subject to its own external data $\bar u_\epsilon$ on
$\R^N\setminus\D^{\hat\delta}$, as:
\begin{equation*}
\begin{split}
\mathcal{A}_\epsilon \bar u_\epsilon(z) & = \frac{1}{2}
\big(\inf_{|y|=1} + \sup_{|y|=1}\big)\fint_{T_\p^{\epsilon,
    \infty}(y)} u_\epsilon\big(z-(x-\bar x)+\hat
z\big)+\xi\;\mathrm{d}\mu_s^N(\hat z) \\ & = u_\epsilon(z-(x-\bar x)) +\xi
= \bar u_\epsilon(z) \qquad \mbox{for all }\; z\in \D^{\hat\delta}.
\end{split}
\end{equation*}
On the other hand, $u_\epsilon$ solves the same problem (with its own external data). Since:
$$u_\epsilon(z)-\bar u_\epsilon(z) = u_\epsilon(z) -
u_\epsilon(z-(x-\bar x))-\xi\leq 0 \qquad\mbox{for all }\; z\not\in \D^{\hat\delta},$$ 
the monotonicity of the solution operator to \ref{dppe} 
(see Theorem \ref{th_exists}) implies that
$u_\epsilon\leq\bar u_\epsilon$ in $\R^N$, so in particular we get:
$u_\epsilon(x) \leq u_\epsilon(\bar x)+\xi$. 
The reverse inequality $u_\epsilon(x) \geq u_\epsilon(\bar x)-\xi$ can
be shown by a symmetric argument. The proof is done. 
\end{proof}

\medskip

We now replace condition (\ref{ecpe1}) by the boundary game regularity
condition in the spirit of condition (\ref{gr1}). More precisely, we say that $x\in\partial\D$ is
game-regular when:
\begin{equation}\label{gr}
\forall \xi,\delta>0\quad\exists \hat\delta,
\hat\epsilon>0\quad\forall \epsilon<\hat\epsilon,~ \bar x\in B_{\hat\delta}(x)\cap\D \quad 
\exists \sigma_{I,0}\quad\forall\sigma_{II}\quad
\PP\big(X_\tau^{\epsilon, x,\sigma_{I,0}, \sigma_{II}}\not\in B_\delta(x)\big)\leq\xi.
\end{equation}

\begin{lemma}\label{gr_then_good}
If every boundary point $x\in\partial\D$ satisfies (\ref{gr}), then
(\ref{ecpe1}) holds for every bounded, uniformly continuous data function $F:\R^N\to \R$.
\end{lemma}
\begin{proof}
Fix $\xi>0$ and choose $\delta>0$ such that:
$$|F(x) - F(\bar x)|\leq\frac{\xi}{3}\qquad\mbox{for all }\; |x-\bar x|<\delta.$$
By (\ref{gr}) there exists $\hat\delta, \hat\epsilon>0$ so that for
all $\epsilon<\hat\epsilon$, $x\in\partial\D$ and $\bar x\in
B_{\hat\delta}(x)$ there exists $\sigma_{I,0}$ with:
$$\PP\big(X_\tau^{\epsilon, x,\sigma_{I,0}, \sigma_{II}}\not\in
B_\delta(x)\big)\leq\frac{\xi}{1+6\|F\|_{L^\infty}}\qquad\mbox{for all }\; \sigma_{II}.$$
Taking $\epsilon, x, \bar x$ as indicated, we obtain:
\begin{equation*}
\begin{split}
u_\epsilon(\bar x) - u_\epsilon(x) & = u^\epsilon_I(\bar x) - F(x) \geq
\inf_{\sigma_{II}}\mathbb{E}\big[ F\circ X_\tau^{\epsilon, x, \sigma_{I,0}, \sigma_{II}}
-F(x)\big] \\ & \geq \mathbb{E}\big[ F\circ X_\tau^{\epsilon, x, \sigma_{I,0}, \sigma_{II,0}} -F(x)\big]-\frac{\xi}{3}
\\ & = \int_{\{X_\tau\in B_\delta(x)\}}F(X_\tau)-F(x)\;\mbox{d}\PP + \int_{\{X_\tau\not\in B_\delta(x)\}}F(X_\tau)-F(x)\;\mbox{d}\PP
-\frac{\xi}{3} \\ & \geq -2\|F\|_{L^\infty}\cdot \PP(X_\tau\not\in
B_\delta(x)) - \frac{\xi}{3} - \frac{\xi}{3} \geq -\xi, 
\end{split}
\end{equation*}
where we used an almost-infimizing strategy $\sigma_{II,0}$. The
inequality $u_\epsilon(\bar x) - u_\epsilon(x)\leq \xi$ follows by a
symmetric argument. This ends the proof.
\end{proof}

\begin{proposition}\label{pomoc_gr}
Fix $\delta>0$, $k\geq 2$, $\epsilon<\frac{\delta}{k}$ and $|x|<\frac{\delta}{k}$.
Then there holds:
$$\forall \sigma_{I},~~\sigma_{II}\qquad
\mathbb{P}\big(|X^{\epsilon, x, \sigma_{I},  \sigma_{II}}_{\bar{\bar \tau}}|\geq
\delta\big)\leq\big(\frac{2}{k-1}\big)^{2s}\doteq a_k,$$ 
where we defined the stopping time: $\bar{\bar\tau}\doteq \min\big\{i\geq 0;~ |X_i|\geq
  \frac{\delta}{k}\big\}$. 
\end{proposition}
\begin{proof}
Let $\delta,k,\epsilon,x$ be as in the statement of the result. It follows that:
\begin{equation*}
\begin{split}
&\mathbb{P}\big(|X_{\bar{\bar \tau}}|\geq\delta\big) \\ & \leq \sup\Big\{
\frac{\mu_s^N\big((x+T_\p^{\epsilon,\infty}(y))\setminus
  B_\delta(0)\big) + \mu_s^N\big((x+T_\p^{\epsilon,\infty}(\bar y))\setminus B_\delta(0)\big)}
{\mu_s^N\big((x+T_\p^{\epsilon,\infty}(y))\setminus
  B_{\delta/k}(0)\big) + \mu_s^N\big((x+T_\p^{\epsilon,\infty}(\bar y))\setminus B_{\delta/k}(0)\big)};
~ |x|<\frac{\delta}{k},~ |y|=|\bar y|=1\Big\} \\ &\leq
\sup\Big\{ \frac{\mu_s^N\big((x+T_\p^{\epsilon,\infty}(y))\setminus
  B_\delta(0)\big)}{\mu_s^N\big((x+T_\p^{\epsilon,\infty}(y))\setminus
  B_{\delta/k}(0)\big)}; ~ |x|<\frac{\delta}{k},~ |y|=1\Big\},
\end{split}
\end{equation*}
where we used the fact that $\frac{\alpha_1+\beta_1}{\alpha_2+\beta_2}\leq\max\big\{
\frac{\alpha_1}{\alpha_2}, \frac{\beta_1}{\beta_2}\big\}$. Further, denoting:
\begin{equation*}
a= \inf_{|x|<\delta/k}\mathrm{dist}(x,\R^N\setminus B_\delta(0)) =
\delta - \frac{\delta}{k},\qquad 
b= \sup_{|x|<\delta/k}\mathrm{dist}(x,\R^N\setminus B_{\delta/k}(0)) = 2 \frac{\delta}{k},
\end{equation*}
leads to:
\begin{equation*}
\mathbb{P}\big(|X_{\bar{\bar \tau}}|\geq\delta\big) \leq
\frac{\sup\Big\{ \mu_s^N\big((x+T_\p^{\epsilon,\infty}(y))\setminus
  B_\delta(0)\big);~ |x|<\frac{\delta}{k},~ |y|=1\Big\} }
{\inf\Big\{ \mu_s^N\big((x+T_\p^{\epsilon,\infty}(y))\setminus
  B_\delta(0)\big);~ |x|<\frac{\delta}{k},~ |y|=1\Big\}  }\leq 
\frac{\mu_s^N(T_\p^{a,\infty})}{\mu_s^N(T_\p^{b,\infty})}.
\end{equation*}
Since $\mu_s^N(T_\p^{a,\infty})= \frac{C(N,s) |A_\p|}{2s
  a^{2s}}$, we obtain that $\mathbb{P}\big(|X_{\bar{\bar
    \tau}}|\geq\delta\big)\leq \big(\frac{b}{a}\big)^{2s}$, as claimed.
\end{proof}

\medskip

Here is the main result of this section:

\begin{theorem}\label{main_equi_thm}
Let $\D\subset\R^N$ be an open, bounded domain satisfying the external
cone condition. Then (\ref{gr}) holds for every $x\in\partial\D$.
\end{theorem}
\begin{proof}
{\bf 1} Fix $\xi,\delta>0$ and $x\in\partial\D$. Without loss of
generality $x=0$. Fix $k_0\geq 4$ such that
$a_{k_0}=\big(\frac{2}{k_0-1} )^{2s}<\frac{\xi}{2}$. It follows by
Proposition \ref{pomoc_gr} that for all $\epsilon<\frac{\delta}{k_0}$ we have:
\begin{equation}\label{ajeden}
\begin{split}
\PP \big(X_{\tau} \not \in B_\delta(0)\big)  & = \PP \Big(\big\{|X_{
    \tau}|\geq\delta\big\}\cap \big\{ \exists n<\tau \quad
|X_n|\in\big [\frac{\delta}{k},\delta\big )\big\} \Big) \\ & \quad
\qquad\qquad + \PP \Big(\big\{|X_{\tau}|\geq\delta\big\}\cap \big\{ \not\exists n<\tau \quad
|X_n|\in\big [\frac{\delta}{k},\delta\big )\big\} \Big) 
\\ & \leq \PP \Big(\exists n<\tau \quad |X_n|\geq\frac{\delta}{k_0}\Big) + \frac{\xi}{2}.
\end{split}
\end{equation}
Denote: $\epsilon_0 = \delta_1=\frac{\delta}{k_0}$. We now show that:
\begin{equation}\label{gr2}
\begin{split}
\exists \hat\delta<\delta_1,~~\hat\epsilon<\epsilon_0\quad\forall \epsilon<\hat\epsilon,~~
|x_0|<\hat\delta\quad\exists\sigma_{I,0}\quad\forall\sigma_{II}\qquad \PP\big(
\exists n<\tau \quad |X_n|\geq \delta_1\big) \leq \frac{\xi}{2}.
\end{split}
\end{equation}
Together with (\ref{ajeden}), (\ref{gr2}) will establish the result.

\smallskip

{\bf 2.} By Proposition \ref{pomoc_gr}, there exists $k\geq k_0$ such
that $\bar\theta + a_k<1$. Let $m\geq 2$ satisfy:
\begin{equation}\label{adwa}
\big(\bar\theta + a_k\big)^m\leq\frac{\xi}{2},
\end{equation}
where $\bar\theta$ is as in (\ref{gr1}).
We now define $\{\epsilon_i\}_{i=1}^m$, $\{\delta_i\}_{i=2}^m$, 
by applying Theorem \ref{cone_thengood} to $\delta_i$ in place of
$\delta$, recursively in:
\begin{equation}\label{set}
\epsilon_i = \min\{\epsilon_{i-1},\hat\epsilon(\delta_i)\},\quad 
\delta_i = \frac{\hat\delta(\delta_{i-1})}{k}.
\end{equation}
We also set:
$$\hat\epsilon\doteq \min\{\epsilon_m, \frac{\delta_m}{k}\},\quad \hat\delta \doteq\hat\delta(\delta_m),
\quad \mbox{ and }\quad  \kappa_i\doteq\min\{j\geq 0;~ |X_j|\geq \delta_i\} \quad \mbox{for all }\;
i=1\ldots m.$$
Given $\epsilon<\hat\epsilon$ and $|x_0|<\hat\delta$, define the strategy $\sigma_{I,0}$ as follows: 
\begin{itemize}[leftmargin=7mm]
\item For $j<\kappa_m$, we utilize the strategy $\sigma_{I,0,m}$ from
  (\ref{gr1}), chosen for the starting point $x_0$:
$$\sigma_{I,0}^{j}\big(x_0, (x_1,z_1,s_1),\ldots, (x_j, z_j,s_j)\big) =
\sigma_{I,0,m}^{j} \big(x_0, (x_1,z_1,s_1),\ldots, (x_j, z_j,s_j)\big). $$
\item If $|X_{\kappa_m}|\geq k\delta_m$, we keep the definition above for all $j\geq \kappa_m$.
\item  If $|X_{\kappa_m}|\in [\delta_m, k\delta_m)$, then for 
  $j\in[\kappa_m, \kappa_{m-1})$ we utilize the strategy $\sigma_{I,0,m-1}$ from
  (\ref{gr1}), chosen for the starting point $X_{\kappa_m}$:
$$\sigma_{I,0}^{j}\big(x_0, (x_1,z_1,s_1),\ldots, (x_j, z_j,s_j)\big) =
\sigma_{I,0,m-1}^{j-\kappa_m} \big(x_{\kappa_m}, (x_{\kappa_m+1},z_{\kappa_m+1},s_{\kappa_m+1}),\ldots, (x_j, z_j,s_j)\big). $$
\item Continue in this fashion, concatenating the strategies
  $\sigma_{I,0,i}$ for the remaining indices $i=m-2,\ldots, 1$. Each
  $\sigma_{I,0,i}$  is chosen from (\ref{gr1}) for the starting point $X_{\kappa_{i+1}}$.
\end{itemize}

\smallskip

By several applications of (\ref{gr1}) and Proposition \ref{pomoc_gr}, it follows that:
\begin{equation*}
\begin{split}
\PP \big(\exists n&<\tau \quad |X_n^{\epsilon, x_0, \sigma_{I,0},
  \sigma_{II}}|\geq\delta_1\big)  =  \PP \big(\kappa_1<\tau\big)  
\\ & = \PP \big(\kappa_2<\kappa_1<\tau \mbox{ and } |X_{\kappa_2}|\in [\delta_2, 4\delta_2)\big) 
+ \PP \big(\kappa_2\leq \kappa_1<\tau \mbox{ and } |X_{\kappa_2}|\geq 4\delta_2\big) 
\\ & \leq \bar\theta \cdot \PP \big(\kappa_2<\tau \mbox{ and } |X_{\kappa_2}|\in [\delta_2, 4\delta_2)\big) 
+ \PP \big(|X_{\kappa_2}|\geq 4\delta_2\big) \cdot \PP \big(\kappa_2<\tau\big) \\ & \leq
\big(\bar\theta + a_k\big) \PP (\kappa_2<\tau).
\end{split}
\end{equation*}
An iteration of the above argument and one final application of (\ref{gr1}) together with (\ref{adwa}) yield:
\begin{equation*}
\begin{split}
\PP \big(\exists n&<\tau \quad |X_n|\geq\delta_1\big)  \leq
\big(\bar\theta + a_k\big)^{m-1} \PP (\kappa_m<\tau) \leq \big(\bar\theta + a_k\big)^{m} 
\leq\frac{\xi}{2}.
\end{split}
\end{equation*}
Thus (\ref{gr2}) has been verified.
\end{proof}

\medskip

By Lemma \ref{gr_then_good}, Lemma \ref{bdary_enough} and invoking the
Ascoli-Arzel\`a theorem, we immediately get, in view of
Theorem \ref{th_viscosity} and Remark \ref{remi_visc}:

\begin{corollary}
Let $\D\subset\R^N$ be an open, bounded domain satisfying the external
cone condition. For a given bounded, uniformly continuous data
function $F:\R^N\to\R$, consider the family $\{u_\epsilon\}_{\epsilon\to 0}$ of solutions to \ref{dppe}.
Then, every sequence $\{u_\epsilon\}_{\epsilon\in J, \epsilon\to 0}$
has a further subsequence that converges uniformly as $\epsilon\to 0$,
to a viscosity solution of (\ref{nD}).
\end{corollary}

\appendix

\section{Case $p_x=0$ and uniqueness of viscosity solutions}\label{appendix}

In this section we extend Definition \ref{def_operators} to the case
of $u\in C^{1,1}(x)$ with $p_x=0$. We then prove uniqueness of
viscosity solutions to (\ref{nD}) under the more restrictive condition
that admits test functions $\phi$ satisfying $\nabla \phi(x)=0$.

\begin{definition}\label{def_operator2}
Given an exponent $\p\in [2,\infty)$, and a bounded, Borel function $u:\R^N\to\R$ such that $u\in
C^{1,1}(x)$, we define:
$$\tLs (x)  \doteq \sup_{|y|=1}\inf_{|\tilde
  y|=1}\int_{T_\p^{0,\infty}(y)}\Lp(x,z,R_{\tilde y, y}z)\dmN.$$
\end{definition}
Note that when $p_x=0$, the above expression is well posed by (\ref{C2}). We also
observe the following easy counterpart of the estimate (\ref{tre}):
\begin{equation*}
\begin{split}
\big|\Lse (x) - \tLs (x)\big|& \leq \sup_{|y|=|\tilde y|=1}\int_{T_\p^{0,\epsilon}(y)}\big|\Lp(x,z,R_{\tilde y, y}z)\big|\dmN
\\ & \leq 2C_x\int_{T_\p^{0,\epsilon}}|z|^2 \dmN = C(N,s) |A_\p|\cdot C_x \frac{\epsilon^{2-2s}}{1-s}.
\end{split}
\end{equation*}
On the other hand, for $p_x\neq 0$ the values of the two operators coincide:
\begin{proposition}\label{Ltilde}
We have $\Ls (x) = \tLs (x)$, whenever both sides are defined.
\end{proposition}
\begin{proof}
For $\delta>0$, let $|y_\delta|=1$ satisfy: $\tLs(x) \leq
\inf_{|\tilde y|=1}\int_{T_\p^{0,\infty}(y_\delta)}\Lp(x,z, R_{\tilde
y, y_\delta}z)\dmN + \delta$. Denote $C=C(C_x, r_x, N, \p,s,\|u\|_{L^\infty}) =
2C_x\int_{T_\p^{0,r_x}}|z|^2\dmN + 4\int_{T_\p^{r_x,\infty}}\|u\|_{L^\infty}\dmN$. Then:
\begin{equation*}
\begin{split}
\tLs (x) & \leq \int_{T_\p^{0,\infty}(y_\delta)}\Lp(x,z,R_{\frac{p_x}{|p_x|},  y_\delta}z)\dmN +\delta
\\ & \leq \int_{T_\p^{0,r_x}(y_\delta)}\big\langle p_x,\big(Id_N - R_{\frac{p_x}{|p_x|},  y_\delta}\big)z\big\rangle
\dmN + C +\delta.
\end{split}
\end{equation*}
Similarly, by (\ref{C2}) it follows that:
\begin{equation}\label{grap}
\tLs (x)\geq \sup_{|y|=1}\inf_{|\tilde y|=1}\int_{T_\p^{0,r_x}(y)}\big\langle p_x,\big(Id_N -
R_{\tilde y, y}\big)z\big\rangle \dmN - C=-C,
\end{equation}
since the first term in the above right hand side equals $0$. We hence arrive at: 
$$ -2 C-\delta\leq \int_{T_\p^{0,r_x}(y_\delta)}\big\langle p_x,\big(Id_N -
R_{\frac{p_x}{|p_x|}, y_\delta}\big)z\big\rangle \dmN
= \left\{\begin{array}{ll} 0 & \mbox{for }  \; y_\delta\neq\frac{p_x}{|p_x|}\\
-\infty & \mbox{otherwise,}  \end{array}\right. $$
so there must be $y_\delta = \frac{p_x}{|p_x|}$. Thus $\tLs(x) \leq
\inf_{|\tilde y|=1}\int_{T_\p^{0,\infty}(\frac{p_x}{|p_x|})}\Lp\big(x,z, R_{\tilde y, \frac{p_x}{|p_x|}}z\big)\dmN$,
and this last inequality is in fact equality in view of the definition of $\tilde{\mathcal{L}}_{s,\p}$.

Let now $|\tilde y_\delta|=1$ be such that: $ \tLs(x) \geq
\int_{T_\p^{0,\infty}(\frac{p_x}{|p_x|})}\Lp\big(x,z, R_{
\tilde y_\delta, \frac{p_x}{|p_x|}}z\big)\dmN - \delta$. As in
(\ref{grap}), there holds: $\tLs(x)\leq C$ so we conclude that:
$$ 2C + \delta\geq \int_{T_\p^{0,r_x}(\frac{p_x}{|p_x|})}\big\langle p_x,\big(Id_N -
R_{\tilde y_\delta, \frac{p_x}{|p_x|}}\big)z\big\rangle \dmN
= \left\{\begin{array}{ll} 0 & \mbox{for }  \; \tilde y_\delta\neq\frac{p_x}{|p_x|}\\
+\infty & \mbox{otherwise.}  \end{array}\right. $$
Thus $\tilde y_\delta=\frac{p_x}{|p_x|}$ for
all $\delta>0$, and $\tLs(x) = \int_{T_\p^{0,\infty}(\frac{p_x}{|p_x|})}\Lp(x,z, z)\dmN$ as claimed.
\end{proof}

Recall that viscosity solutions to $\tilde{\mathcal{L}}_{s,\p}[u]=0 $ are defined as in Remark \ref{remi_visc},
but without requesting that $\nabla \phi(x)\neq 0$. More precisely, we say that a bounded and uniformly
continuous function $u:\R^N\to \R$ is a viscosity
subsolution, provided that for every $x\in \D$, $r>0$ and every
$\phi\in C^2(\R^N)$, conditions $\phi>u$ on $\bar B_r(x)\setminus \{0\}$
and $\phi(x) = u(x)$, imply that:
$$\tilde{\mathcal{L}}_{s,\p}\big[\one_{\bar B_r(x)}\cdot \phi +
\one_{\R^N\setminus\bar B_r(x)}\cdot u \big](x)\geq 0.$$
When $-u$ is a viscosity subsolution, then $u$ is said to be a viscosity supersolution.
The main result of this section is the following comparison principle:
\begin{theorem}\label{th_compar}
Let $v,w:\R^N\to \R$ be two uniformly continuous, bounded
functions, and let $\D\subset \R^N$ be an open, bounded domain. Assume
that $v$ is a viscosity subsolution while $w$ is a viscosity
supersolution to $\tilde{\mathcal{L}}_{s,\p}[u]=0$ in $\D$. Then 
$v\leq w$ in $\R^N\setminus\D$ implies $v\leq w$ in $\D$.
\end{theorem}

From Proposition \ref{Ltilde} there immediately follows:
\begin{corollary}\label{unique}
Given a bounded and uniformly continuous datum $F:\R^N\to\R$, 
there exists at most one viscosity solution $u$ to the Dirichlet problem:
$$\tilde{\mathcal{L}}_{s,\p}[u]=0 \; \mbox{ in }\;\D, \; \qquad u=F \; \mbox{ in }\;\R^N\setminus\D.$$
\end{corollary}

\medskip

\noindent {\bf Proof of Theorem \ref{th_compar}}

{\bf 1.} For every $\epsilon>0$, we define the sup- and
inf-convolution functions $v^\epsilon, w_\epsilon:\R^N\to \R$ by:
$$v^\epsilon(x) \doteq \sup_{y\in\R^N}\Big\{v(y) + \epsilon -
\frac{|y-x|^2}{\epsilon}\Big\}, \quad w_\epsilon(x) \doteq
-(-w)^\epsilon(x) = \inf_{y\in\R^N}\Big\{w(y) - \epsilon + \frac{|y-x|^2}{\epsilon}\Big\}.$$
Clearly: $v^\epsilon(x) \geq v(x) + \epsilon$. As shown in
\cite{CC}, the continuous functions $\{v^\epsilon\}_{\epsilon\to 0}$ form a nonincreasing sequence
converging uniformly to $v$. We recall that for every $x\in\R^N$, $\epsilon>0$ there exists
a point $x_\epsilon^*\in \R^N$ such that:
\begin{equation} \label{grro} 
v^\epsilon(x) = v(x_\epsilon^*) + \epsilon - \frac{|x-x_\epsilon^*|^2}{\epsilon}.
\end{equation}
Thus indeed $v^\epsilon\searrow v$ uniformly, because $0<v^\epsilon(x) - v(x)\leq v(x_\epsilon^*) -
v(x) + \epsilon$ and:
\begin{equation}\label{grro0}
|x-x_\epsilon^*|^2 = \epsilon\big(v(x_\epsilon^*)+\epsilon -
v^\epsilon(x)\big) \leq \epsilon\big(v(x_\epsilon^*)-v(x)\big) \leq
2\epsilon \|v\|_{L^\infty},
\end{equation}
and because $v$ is uniformly continuous. Similar assertions are valid
for the sequence $w_\epsilon\nearrow w$.

\smallskip

{\bf 2.} Suppose that, contrary to the statement of the theorem, $v(x_0) - w(x_0)=c>0$ for
some $x_0\in\D$. For $\epsilon\ll 1$ each continuous function $v^\epsilon-w_\epsilon$
attains then its maximum value $\delta_\epsilon\geq c$ at some
$\bar x(\epsilon)\in\D$. We will show that for sufficiently small
$\epsilon$ this brings a contradiction.

We fix $\epsilon\ll 1$ and write $\bar x =\bar x(\epsilon)$, $\delta=\delta_\epsilon$. 
Then $v^\epsilon(\bar x) = w_\epsilon(\bar x)+\delta$ and
$v^\epsilon\leq w_\epsilon+\delta$ in $\R^N$. Consider the paraboloids
$P_v, P_w:\R^N\to\R$ given by:
$$P_v(x) \doteq v(\bar x_\epsilon^*)+\epsilon - \frac{|x-\bar
  x_\epsilon^*|^2}{\epsilon}, \qquad P_w(x) \doteq w(\bar{\bar
  x}_\epsilon^*)+\epsilon - \frac{|x-\bar{\bar x}_\epsilon^*|^2}{\epsilon} + \delta,$$ 
where ${\bar x}_\epsilon^*,\bar{\bar x}_\epsilon^*$ are as in
(\ref{grro}) in correspondence to $\bar x, v$ (for $\bar
x_\epsilon^*$), and to $\bar x, w$ (for $\bar{\bar
x}_\epsilon^*$), respectively. These paraboloids have the property that $P_v\leq v^\epsilon$ and
$P_w\geq w_\epsilon+\delta$ in $\R^N$, whereas $P_v(\bar x) =
v^\epsilon(\bar x)$ and $P_w(\bar x) = w_\epsilon(\bar x)+\delta$. It
follows that $v^\epsilon$ and $w_\epsilon$ are both
$C^{1,1}(\bar x)$ with the same gradient:
$$p_{\bar x} \doteq \nabla P_v(\bar x) = \nabla P_w(\bar x).$$

Define now a $C^2$ function $\phi:\bar B_1(\bar x)\to \R$ with the properties:
\begin{equation}\label{grro1}
\begin{split}
& \phi(x) = P_w(x)+\epsilon|x-\bar x|^2\qquad\qquad \qquad\quad \qquad \mbox{ for }\; |x-\bar
x|\leq r_1\doteq \epsilon^{1+\frac{1}{2-2s}}\\
& \phi(x) \in\big(v^\epsilon(x), v^\epsilon(x)+2\max_{\bar B_{r_1}(\bar
  x)}(P_w-P_v)\big) \qquad \,\mbox{ for }\; |x-\bar x|\in (r_1,1).
\end{split}
\end{equation}
The translated function $\displaystyle \psi(x) \doteq \phi(x+\bar x-\bar
x_\epsilon^*)-\epsilon+\frac{|\bar x-\bar x_\epsilon^*|^2}{\epsilon}$ then satisfies:
\begin{equation*}
\begin{split}
& \psi(\bar x_\epsilon^*)  = P_w(\bar x) - \epsilon +
\frac{|\bar x-\bar x^*_\epsilon|^2}{\epsilon} = v^\epsilon(\bar
x) -\epsilon+ \frac{|\bar x-\bar x^*_\epsilon|^2}{\epsilon} = v(\bar x^*_\epsilon),\\
& \psi(x) = P_w(x+\bar x-\bar x_\epsilon^*) +\epsilon|x-\bar x_\epsilon^*|^2-\epsilon +
\frac{|\bar x-\bar x^*_\epsilon|^2}{\epsilon} > w_\epsilon(x+\bar
x-\bar x_\epsilon^*) +\delta -\epsilon+ \frac{|\bar x-\bar
  x^*_\epsilon|^2}{\epsilon} \\ & \qquad \geq
v^\epsilon(x+\bar x-\bar x_\epsilon^*) -\epsilon+ \frac{|\bar x-\bar x^*_\epsilon|^2}{\epsilon}
\geq v(x)\qquad\mbox{ for all } \; x\in \bar B_1(\bar x_\epsilon^*).
\end{split}
\end{equation*}
Hence, $\psi$ can be used in the definition of the viscosity subsolution,
as supporting $v$ from above at $\bar x_\epsilon^*\in\D$ and relative to the radius $r=1$.
By assumption, this implies:
$$\tilde{\mathcal{L}}_{s,\p}\Big[\one_{\bar
  B_1(\bar x_\epsilon^*)}\cdot\phi + \one_{\R^N\setminus \bar B_1(\bar
  x_\epsilon^*)}\cdot v\Big](\bar x_\epsilon^*)\geq 0.$$

\smallskip

{\bf 3.} Through a straightforward change of variables, the above inequality becomes:
\begin{equation}\label{grro2}
\begin{split}
\sup_{|y|=1}\inf_{|\tilde y|=1}\Big(&\int_{T_\p^{0,1}(y)}\phi(\bar
x+z) + \phi(\bar x- R_{y,\tilde y}z) - 2v^\epsilon(\bar x)\dmN \\ &
+ \int_{T_\p^{1,\infty}(y)} v(\bar x_\epsilon^*+z) + v(\bar x_\epsilon^* -
R_{y,\tilde y}z) - 2v^\epsilon(\bar x)\dmN \Big) \geq 0.
\end{split}
\end{equation}
From (\ref{grro1}) and (\ref{grro0}) we observe that:
\begin{equation*}
\begin{split}
& \phi(x) - v^\epsilon(x) \leq P_w(x)-P_v(x) + \epsilon|x-\bar x|^2
\leq \big(\frac{2}{\epsilon} + \epsilon\big)|x-\bar x|^2\leq
\frac{3}{\epsilon}|x-\bar x|^2\quad \mbox{for }\; |x-\bar x|\leq r_1,\\
& \phi(x) - v^\epsilon(x) \leq 2\max_{\bar B_{r_1}(\bar
  x)}(P_w-P_v)\leq \frac{4}{\epsilon} r_1^2\qquad
\quad\mbox{for }\; |x-\bar x|\in (r_1,1],\\
& v(x) - v^\epsilon(x+\bar x-\bar x_\epsilon^*)\leq v(x) - v(x-(\bar
x_\epsilon^* - \bar x))\leq \omega_v(2\epsilon\|v\|_{L^\infty}) \quad \mbox{for all }\; x\in\R^N,
\end{split}
\end{equation*}
where $\omega_v$ denotes the modulus of continuity of $v$.
Replacing $\phi$ and $v$ in (\ref{grro2}) by $v^\epsilon$ implies:
\begin{equation*}
\begin{split}
0 & \leq \tilde{\mathcal{L}}_{s,\p}\big[ v^\epsilon\big](\bar x) +
\int_{T_\p^{0,r_1}}\frac{6}{\epsilon}|z|^2\dmN
+\frac{8}{\epsilon} r_1^2 \cdot \mu_s^N(T_\p^{r_1,1}) + 2 \omega_v(2\epsilon\|v\|_{L^\infty}) 
\cdot \mu_s^N(T_\p^{1,\infty}) \\ & \leq \tilde{\mathcal{L}}_{s,\p}\big[ v^\epsilon\big](\bar x) +
C(N,s) \Big(\frac{3}{\epsilon}\frac{r_1^{2-2s}}{1-s} +
\frac{4}{\epsilon}r_1^2\cdot \frac{r_1^{-2s}-1}{2s} + \frac{\omega_v(2\epsilon\|v\|_{L^\infty})}{s}\Big)\\
& = \tilde{\mathcal{L}}_{s,\p}\big[ v^\epsilon\big](\bar x) +
C(N,s) \Big(4\big(\frac{1}{1-s}+\frac{1}{s}\big)\epsilon^{2-2s} +
\frac{4}{s}\epsilon^{1+\frac{1}{1-s}} + \frac{\omega_v(2\epsilon\|v\|_{L^\infty})}{s}\Big).
\end{split}
\end{equation*}
Applying the same reasoning above to the approximation $w_\epsilon$, we can conclude that:
\begin{equation*}
\begin{split}
0 & \geq \tilde{\mathcal{L}}_{s,\p}\big[ w_\epsilon+\delta\big](\bar x) -
C(N,s) \Big(4\big(\frac{1}{1-s}+\frac{1}{s}\big)\epsilon^{2-2s} +
\frac{4}{s}\epsilon^{1+\frac{1}{1-s}} + \frac{\omega_w(2\epsilon\|w\|_{L^\infty})}{s}\Big).
\end{split}
\end{equation*}
In particular, there further follows:
\begin{equation}\label{grro3}
\tilde{\mathcal{L}}_{s,\p}\big[ w_\epsilon+\delta\big](\bar x) 
- \tilde{\mathcal{L}}_{s,\p}\big[ v^\epsilon\big](\bar x)\to 0 \quad \mbox{ as }\; \epsilon\to 0.
\end{equation}

On the other hand, one directly sees that:
\begin{equation*}
\begin{split}
\tilde{\mathcal{L}}_{s,\p}&\big[w_\epsilon+\delta\big](\bar x) - \tilde{\mathcal{L}}_{s,\p}\big[v^\epsilon\big](\bar x) \geq
\inf_{|y|=|\tilde y|=1} \int_{T_\p^{0,\infty}(y)}
L_{w_\epsilon-v^\epsilon +\delta}\big(\bar x, z, R_{\tilde y,
  y}z\big)\dmN \\ & \geq 2 \inf_{|y|=1}\int_{T_\p^{0,\infty}(y)}
\big(w_\epsilon-v^\epsilon+\delta\big)(\bar x+z)\dmN \\ & \geq 
2 \big(c+\inf_{\R^N\setminus\D}(w_\epsilon - v^\epsilon)\big)\cdot\mu_s^N\big(T_\p^{\mathrm{diam}\D,\infty}\big)
\geq c \cdot \mu_s^N\big(T_\p^{\mathrm{diam}\D,\infty}\big),
\end{split}
\end{equation*}
in contradiction with (\ref{grro3}). The last inequality displayed above holds for $\epsilon>0$ sufficiently small to
guarantee that $\inf_{\R^N\setminus\D}(w_\epsilon -
v^\epsilon)\geq-\frac{c}{2}$, in view of the assumption $w-v\geq 0$ on $\R^N\setminus\D$.
The proof is complete.
\endproof

\end{document}